\documentclass[10pt]{amsart}

\usepackage{amsmath, amsthm, amssymb, graphicx, color}
\usepackage[margin=4cm]{geometry}

\allowdisplaybreaks[2]

\newcommand{\R}{\mathbb{R}}
\newcommand{\N}{\mathbb{N}}

\newcommand{\SL}{\mathcal{L}}

\newcommand{\IP}[2]{\left<#1,#2\right>}
\newcommand{\vn}[1]{\lVert#1\rVert}

\newcommand{\rd}[2]{\frac{d#1}{d#2}}

\newcommand{\Ko}{K_{\text{$\mspace{-1mu}o\mspace{-1mu}s\mspace{-1mu}c$}}}

\newcommand{\kav}{\overline{k}}
\newcommand{\gammaiav}{\overline{\gamma_i}}

\newtheorem{thm}{Theorem}[section]
\newtheorem*{thm*}{Theorem}
\newtheorem*{QN}{Question}
\newtheorem*{conj}{Conjecture}
\newtheorem*{prop*}{Proposition}
\newtheorem*{cor*}{Corollary}
\newtheorem*{lem*}{Lemma}
\newtheorem{prop}[thm]{Proposition}
\newtheorem{lem}[thm]{Lemma}
\newtheorem{cor}[thm]{Corollary}

\theoremstyle{remark}
\newtheorem{rmk}{Remark}
\begin{document}

\title{Curve diffusion and straightening flows on parallel lines}

\author{Glen Wheeler$^*$ and Valentina-Mira Wheeler}
\thanks{*: Corresponding author.}
\address{Glen Wheeler\\
           Institute for Mathematics and its Applications \\
           University of Wollongong\\
           Northfields Avenue\\
           Wollongong, NSW, 2522, Australia\\
           email: glenw@uow.edu.au }
\address{ Valentina-Mira Wheeler \\
           Institute for Mathematics and its Applications \\
           University of Wollongong\\
           Northfields Avenue\\
           Wollongong, NSW, 2522, Australia\\
           email: vwheeler@uow.edu.au
           }

\subjclass[2000]{53C44 \and 58J35}

\begin{abstract}
In this paper, we study families of immersed curves
$\gamma:(-1,1)\times[0,T)\rightarrow\R^2$ with free boundary supported on
parallel lines $\{\eta_1, \eta_2\}:\R\rightarrow\R^2$ evolving by the curve
diffusion flow and the curve straightening flow.  The evolving curves are
orthogonal to the boundary and satisfy a no-flux condition.  We give estimates
and monotonicity on the normalised oscillation of curvature,
yielding global results for the flows.
\end{abstract}
\maketitle

\section{Introduction}

Fourth-order extrinsic curvature flow have recently enjoyed considerable attention in the literature.
Two model flows are the surface diffusion flow, where points move with velocity $\Delta^\perp \vec{H}$, and
the Willmore flow, where points move with velocity $\Delta^\perp \vec{H} + \vec{H}|A^o|^2$.
These curvature flow are one-parameter families of surfaces immersed in $\R^3$ via immersions
$f:\Sigma\times[0,T)\rightarrow\R^3$, with $\vec{H}$ the mean curvature vector, $\Delta^\perp$ the Laplacian
on the normal bundle along $f$, and $A^o$ the tracefree second fundamental form.

Surface diffusion flow, proposed by Mullins \cite{Mullins} in 1956, arises as a model for several phenomena
\cite{CTNc96,CT94}.
As such it has received and continues to receive intense attention from the applied mathematics community.
Global analysis for the surface diffusion flow is restricted at the moment to special situations, and although
the theory of singularities for the flow has received some attention \cite{W09,W10} it is far from
well-understood.
The surface diffusion flow is variational, being in a sense the $H^{-1}$-gradient flow for the area functional.
The Willmore flow is also variational, being the steepest descent $L^2$-gradient flow for the Willmore
functional.
The Willmore functional is, up to normalisation, the integral of the mean curvature $\vec{H}$ squared.
A prototypical bending energy, it has been argued that the Willmore functional was considered first by Sophie
Germain in the early 19th century.
The Willmore functional drew significant interest from Blaschke \cite{blaschke3,blaschke2,blaschke1} and his
school, including Thomsen and Schadow, who first presented the Euler-Lagrange operator.
Their interest in the Willmore functional stems from its invariance under the M\"obius group of $\R^3$ (so
long as inversions are not centred on the surface, see \cite{BK,bernardQ,Chen1974,kusner89} for example for a
precise formula).
This invariance lies at the heart of many of its applications, both to physics and back to mathematics itself,
for example in embedding problems.
The appeal of the functional is so universal that the Willmore conjecture \cite{Wconj}, asserting that the
global minimiser among surfaces in $\R^3$ with genus one is achieved by the Clifford torus (and closed
conformal images thereof), generated significant attention (a selection is \cite{chen1970,LY82,ros99,S93}),
before being recently solved in a brealthrough work \cite{MN14}.
The Willmore flow was first studied by Kuwert and Sch\"atzle \cite{KS01,KS02,KS04} who set up a general
framework that is by now a standard methodology used to understand large varieties of higher-order curvature
flow.
Applications and modifications of this framework exist for the surface diffusion flow \cite{W10,W13}, the
geometric triharmonic heat flow \cite{WMP15}, Chen's flow \cite{BWW,CWW} and polyharmonic flows \cite{WP15,PW19}.

Although in some special cases maximum-principle style results hold, more typical is a kind of `almost'
maximum principle, and an `eventual' positivity, see \cite{daners2015,FGG08,GG08,GG09} for the parabolic and
\cite{grunaubook} and for the elliptic settings respectively).
Many of the tools and techniques used in the analysis of second-order curvature flow can not be applied to the
study of fourth and higher-order curvature flow.
In addition to the development of new techniques, it is a natural focus of research effort to determine the
extent to which modifications of known techniques apply to various fourth-order curvature flow in different
scenarios.
This is where the present paper fits into the picture.
We treat the one-dimensional case for the surface diffusion and Willmore flows
with free boundary, called the \emph{curve diffusion flow} and \emph{free elastic
flow} (or simply the \emph{elastic flow}) respectively.

In order to differentiate easily between these two flows, we label them as follows:
\vspace*{-1mm}
\begin{itemize}
\item[(CD)] Curve diffusion flow
\item[(FE)]  Free elastic flow
\end{itemize}  
The main results, Theorem \ref{TMmainb} for (CD) and Theorem
\ref{TMmainbElastic} for (FE), consider the question of geometric stability,
where closeness to an equilibrium is measured explicitly in terms of a
geometric quantity.
We also present some conjectures and a question on a suitable adaptation of
Proposition 1.5 from \cite{W13}.  This directly addresses for (CD) the question
of preservation of positivity raised above by measuring the total amount of
time during which a global solution may remain not strictly graphical.
The evolving families of planar curves we study are line segments meeting a
pair of parallel lines at right angles (see Figure \ref{Fig1}).

Second-order curvature flow with free boundary have been considered since the
90s \cite{pihan,S94,ST2,ST1} and continues to receive significant research
attention (for a sample of the growing literature, see
\cite{buckland,depner14,koeller12,lambert,marquardt1,marquardt2,mizuno15,Volkmann,vulcanov,WW14,wheelerV,V14,V14r}).
Fourth-order curvature flow with various boundary conditions have received some
recent attention, with work particularly relevant to this paper in
\cite{dall2014a,dall2014,dalllojasiewicz,GIK05,GIK08,lin2012,lin2015,mccoywheelerwusixth,mccoywheelerwuhighorder,novaga2014curve,oelz2011}.
In \cite{GIK05,GIK08} stability results are proved for curves evolving by (CD)
that are graphical and nearby equilibria (with closeness measured in terms of
height and $\vn{k_s}_2^2$) evolving in bounded domains with free boundary.
Although our setting is fundamentally parametric and therefore distinct, our
results here, for the curve diffusion flow, can be thought of as naturally
complementing these.
The evolving curves considered in this paper are supported on straight lines,
so the analogue of `domain' from \cite{GIK05,GIK08} is always unbounded.
We consider immersed curves, with possibly self-intersecting image.
Intersections in the image may result from the curve touching itself, or from
the curve intersecting one of the straight supporting lines.
This allows global results for perturbations of arcs of multiply-covered
circles for instance.
Considering curves supported on parallel lines allows for results on unbounded,
cocompact initial data as well. As the supporting curves are parallel, repeated
reflection produces an entire curve.

Stability for the elastic flow is a classically difficult problem.
The flow (FE) is the steepest descent $L^2$-gradient flow for the
elastic energy:
\[
E(\gamma) = \frac12\int_\gamma k^2\,ds\,,
\]
where $\gamma:[-1,1]\rightarrow\R^2$ is a smooth immersed plane curve, $k$ its
scalar curvature and $ds$ the arc-length element.
This energy is \emph{not} scale-invariant, and can be decreased by enlarging
the curve through homothety.
Circles and curves with constant curvature are not equilibria; they are
expanders.

There exist infinitely many straight line segments that are stationary under
the flow.
It seems difficult to imagine that the flow (FE) without a
constraint would be stable, especially without imposing an additional symmetry
condition, as glued in arcs of circles would still prefer to expand under the
flow.
For the flow of closed curves, trajectories are not bounded.
Indeed, if the distance between the parallel lines $|e|$ is zero, then circles
are permitted and these expand.
By slowly separating the two lines (continuously increasing $|e|$ for
example), it seems intuitive that there would exist non-compact trajectories
for the flow (perhaps by some continuous dependence result).
With this in mind, stability of the straight line under (FE) seems unlikely.
Nevertheless we do manage to prove convergence for (FE) under a curvature condition
without needing to resort to a length constraint or penalisation in the energy.
Our convergence result does not require reparametrisation nor translations to
fix a point.
Therefore it firmly states that \emph{straight lines are stable} for the free
boundary (free) elastic flow.

Let us formally introduce the evolution equations.
Suppose $\gamma:[-1,1]\rightarrow\R^2$, $\eta_i:\R\rightarrow\R^2$ ($i = 1$,
$2$) are regular smooth immersed plane curves such that $\gamma$ meets each $\eta_i$
perpendicularly with zero flux at its endpoints; that is,
\begin{equation}
\label{EQbcs}
\gamma(-1) \in \eta_1(\R)
\,,\quad
\gamma(1) \in \eta_2(\R)
\,,\quad
\IP{\nu}{\nu_{\eta_i}}(\pm1) = 0
\,,\quad
k_s(\pm1) = 0
\,.
\end{equation}
Above we have used $\nu$ to denote a unit normal vector field on $\gamma$, 
the subscript $s$ to denote application of $\partial_s =
\frac{1}{|\gamma_x|}\partial_x$ where $x$ is the variable in the given
parametrisation of $\gamma$, and $k = \IP{\kappa}{\nu} =
\IP{\gamma_{ss}}{\nu}$.
Note that we are not using $s$ here as a true parameter.
We choose $\nu$ by setting $\nu = (\tau_2, -\tau_1)$ where $\tau = \gamma_s$ is
the tangent vector with direction induced by the given parametrisation.
We call $\eta_i$ \emph{supporting curves} for the flow.

The length of $\gamma$ is
\[
L(\gamma) = \int_{-1}^1 |\gamma_x|\, dx\,.
\]
Another important quantity, in addition to the elastic energy $E$ introduced
earlier, is
\begin{equation}
A(\gamma) = -\frac{1}{2}\int_{-1}^1 \IP{\gamma}{\nu}\,|\gamma_x|\,dx\,,
\label{AF}
\end{equation}
which is the usual notion of area for closed plane curves.
Here, $A$ corresponds to the area of the star-shaped domain (with multiplicity)
traced out by rays connecting the position vector $\gamma$ and the origin.

Consider a one-parameter family of immersed curves
$\gamma:[-1,1]\times[0,T)\rightarrow\R^2$ satisfying the boundary conditions
\eqref{EQbcs} and with normal speed given by $-F$, that is
\begin{equation*}
\partial_t\!\gamma = -F\nu\,.
\end{equation*}
The flows are:
\begin{description}
\item[(CD)]
Curve diffusion, where the normal velocity is equal to $-\text{grad}_{H^{-1}(ds)}(L(\gamma))$\footnote{In a sub-Riemannian horizontal graphical sense.}, that is,
\[
F = k_{ss}\,;
\]
\item[(FE)]
Free elastic flow, where the normal velocity equal to $-\text{grad}_{L^{2}(ds)}(E(\gamma))$, that is,
\[
F = k_{ss} + \frac12k^3\,.
\]
\end{description}
The (free) boundary value problem that we wish to consider for these flows is the following:
\begin{equation}
\label{CD}
\tag{CD/FE}
\begin{cases}
(\partial_t\gamma)(x,t) = -(F\nu)(x,t)&\text{ for all }(x,t)\in(-1,1)\times(0,T)\hspace{-1cm}
\\
\gamma(-1,t) \in \eta_1(\R);\quad \gamma(1,t) \in \eta_2(\R)\quad&\text{ for all }t\in\times[0,T)
\\
\IP{\nu}{\nu_{\eta_1}}(-1,t) =
\IP{\nu}{\nu_{\eta_2}}(1,t) = 0\quad&\text{ for all }t\in\times[0,T)
\\
k_s(-1,t) = k_s(1,t) = 0\qquad\qquad\qquad&\text{ for all }t\in\times[0,T)
\,.
\end{cases}
\end{equation}

The curve diffusion flow is in a sense the steepest descent gradient flow for length in $H^{-1}$.
Since the velocity is a potential, signed enclosed area $A$ in the case of
closed curves is constant along the flow.
This shows that the isoperimetric ratio is a scale-invariant monotone quantity
for the flow, and this fact can be useful for analysis of solutions to the flow
(see \cite{W13} for example).
In the case of the boundary problems considered here, this is no longer true.
Here it is difficult to find a useful notion of enclosed area.
Indeed, this is a fundamental obstacle to smooth compactness, and can only be
overcome in the case when the flow is already in its preferred topological
class, that is, when we assume that $\omega = 0$ (see Remark
\ref{RMKfinitetime} and Figure \ref{FigExterior}).

Local existence for \eqref{CD} can be proved by using the standard procedure of
solving the flow in the class of graphs over the initial data, as in
\cite{S94}.
As we consider a Neumann problem, we may use a local adapted coordinate system
similar to Stahl \cite{S94} which does not require a tangential component in
the velocity of the flow.
This can be continued until the solution leaves this class, at which point
there is either some loss of regularity in $C^{4,\alpha}$, or the solution is
simply no longer graphical over its initial state.
The latter problem is a technicality, and can be resolved by writing the flow
in a new coordinate system, as a graph over the solution at a later time.
Now if there are uniform $C^{2,\alpha}$-estimates, it is possible to use a
standard contraction map argument to continue the solution. To the best of our
knowledge the first to observe that only $C^{2,\alpha}$ is required were
Ito-Kohsaka, with the map $\Phi$ constructed in \cite[Proof of Theorem
3.1]{IK01}.
Their proof is written for the flow (CD), however we note that the technique applies also to the flow (FE)
without significant difficulty.
Therefore by iterating the above procedure we find that the maximal time of
existence is either infinity, or the $C^{2,\alpha}$ norm has blown up.
In this paper, the most natural norms to control a-priori are $L^2$ norms of $k$, $k_s$, \ldots, $k_{s^p}$ and so on.
The standard Sobolev inequality allows us
to control the $C^{2,\alpha}$ norm by the length of the position vector
$|\gamma|$ and the $L^2$-norm of the first derivative of curvature.
Note that it is not (without additional arguments) enough to bound only the
length of the evolving curves, pointwise control on the position vector is needed.
The statement below is specialised to our current situation, where the
supporting curves are straight lines.  We note that it is not optimal.

\begin{thm}[Local existence]
\label{TMste}
Let $\eta_i:\R\rightarrow\R^2$ ($i = 1$, $2$) be straight lines.
Suppose $\gamma_0:\R\rightarrow\R^2$ is a regular smooth curve satisfying the
boundary conditions \eqref{EQbcs}.
Then there exists a maximal $T\in(0,\infty]$ and a unique one-parameter family
of regular immersed curves $\gamma:(-1,1)\times[0,T)\rightarrow\R^2$ satisfying
$\gamma(x,0) = \gamma_0(x)$ and \eqref{CD}.
Furthermore, if $T<\infty$, then there does not exist a constant $C$ such that
\begin{equation}
\label{EQextension}
\vn{\gamma}_\infty + \vn{k_s}_2 \le C
\end{equation}
for all $t\in[0,T)$.
\end{thm}

\begin{rmk}
If the flow is not supported on straight lines, then we require compatibility conditions to produce a
solution.
If the compatibility conditions are violated by the initial data, then we are still typically able to produce
a flow, however convergence as $t\searrow0$ will be limited by the degree to which the compatibility
conditions are satisfied.
One interesting investigation into this for the surface diffusion flow is \cite{GA}, where the degree of
incompatibility is finely studied in the context of the original motivation from Mullins \cite{Mullins}.
\end{rmk}


\subsection{Curve diffusion flow}

In light of condition \eqref{EQextension}, global existence follows if we are
able to uniformly bound the position vector $\gamma$ and the $L^2$-norm of
$k_s$.
The curve diffusion flow is the $H^{-1}$ gradient flow of the length
functional, with $L' = -\vn{k_s}_2^2$.
The length is uniformly controlled a-priori but this does not yield immediately
an estimate for $\vn{\gamma}_\infty$.
It does make $\vn{k_s}^2_2$ a natural energy for the flow, with an a-priori
uniform estimate in $L^1([0,T))$ in terms of the length of the initial
data.
Despite this, there are shrinking self-similar solutions to the evolution
equation (see Figure \ref{FigLemniscate}, which relies upon the lemniscate
described in \cite{edwards14}) that are clearly singular in finite-time.
There is a conjecture due to Giga (see \cite{GWconvcdf} for a discussion) that all singularities flow from immersed (and not embedded) data.
For our situation here,
we expect that there exist a greater variety of singularities.

\begin{figure}
\includegraphics[trim=5cm 11cm 5cm 11cm,clip=true,width=6cm,height=4cm]{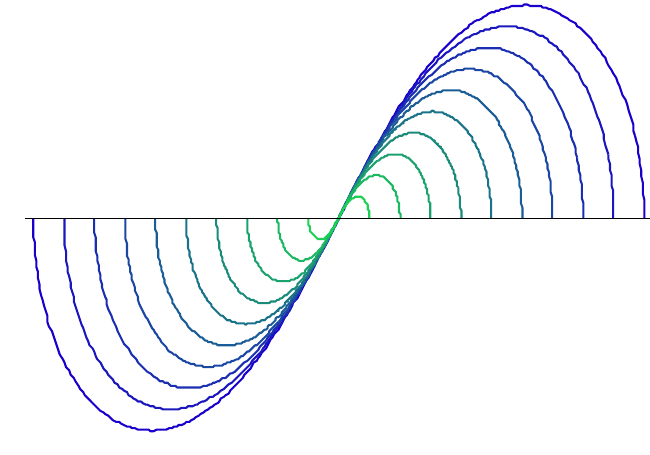}
\caption{The curve diffusion flow with free boundary becoming singular in
finite time. The figure satisfies the no-flux boundary conditions on the black line.
In this case, $|e| = 0$.
The evolution is homothetic.}
\label{FigLemniscate}
\end{figure}

Therefore global existence is not expected to hold generically.
It is natural to hope however that in a suitable neighbourhood of minimisers
for the energy, global existence and convergence to a minimiser holds.
The only global minimisers are straight lines perpendicular to
the supporting lines.
Our main theorem confirms that these equilibria are stable, with neighbourhood
given by the oscillation of curvature.

First let us define the following:
\begin{itemize}
\item Set $e$ to be any vector in the plane perpendicular to each of the
parallel lines $\eta_i$ with length equal to the distance between $\eta_1$ and $\eta_2$.
When we refer to a \emph{horizontal} line segment, we mean a line segment that
is parallel to $e$.
\item The constant $\omega$ and the average of the curvature scalar are defined by
\[
\int_{\gamma} k\,ds\bigg|_{t=0} = 2\omega\pi\,,
\]
\[
\kav (\gamma)= \frac{1}{L}\int_{\gamma} k\, ds\,.
\]
Note that $\omega$ is not typically an integer (see Lemma \ref{WN}).

\item The oscillation of curvature and the isoperimetric ratio are defined as
\[
\Ko (\gamma) = L\int_{\gamma} \big(k-\kav\big)^2 ds\,,
\]
and
\[
I(\gamma) = \frac{L^2(\gamma)}{4\omega\pi A(\gamma)}\,.
\]
\end{itemize}

\begin{thm}
\label{TMmainb}
Suppose $|e|>0$.
Let $\gamma:(-1,1)\times[0,T)\rightarrow\R^2$ be a solution to (CD).
Suppose $\gamma_0$ satisfies
\begin{equation}
\label{EQhyp}
L(0)\vn{k}_2^2(0) < \frac{2\pi}{9}
\,.
\end{equation}
Then $\omega = 0$, the flow exists globally $T=\infty$, and $\gamma(\cdot,t)$
converges exponentially fast to a straight line segment parallel to $e$ in the
$C^\infty$ topology.
\end{thm}

\begin{rmk}
\label{RMKomegazero}
The hypothesis of Theorem \ref{TMmainb} implies that $\omega = 0$.
To see this, we calculate at initial time
\[
(2\omega\pi)^2 = \bigg(\int_\gamma k\,ds\bigg)^2 \le L\int_\gamma k^2\,ds
 < \frac{2\pi}{9}
\]
so
\begin{equation}
\label{EQwindest}
\omega^2 < \frac{2\pi}{9}\frac{1}{4\pi^2} < \frac14\,.
\end{equation}
Now, the boundary condition implies that the total
accumulated angle $\theta(x)$ that the tangent vector makes with $e$ at $x=-1$
is either $0$ or $\pi$, and at $x=1$ it is an integer multiple of $\pi$.
So, we may assume $\theta(-1) = n\pi$ and $\theta(1) = N\pi$ where $n,N$ are integers.
Then the fact that $k = \theta_s$ implies 
\begin{equation}
\label{EQwinding}
\omega = \frac{1}{2\pi}\int_\gamma \theta_s\,ds = \frac{1}{2\pi}\Big( \theta(1) - \theta(-1) \Big) = \frac12(N-n)
\,.
\end{equation}
So we see that $\omega$ is a multiple of $\frac12$.
The estimate \eqref{EQwindest} implies that it must be zero.
As $\omega$ is constant along the flow (see Lemma \ref{WN}), it
remains zero for all time. 
\end{rmk}

\begin{rmk}
Identifying which straight line segment the solution converges to is a
difficult open problem, similar to the problem of identifying the location of
the final point that an embedded planar curve shortening flow approaches (see
\cite{bryantgriffiths}).
\end{rmk}

\begin{figure}
\begin{tabular}{ccc}
 &
\includegraphics[trim=5cm 10cm 3cm 9cm,clip=true,width=5cm]{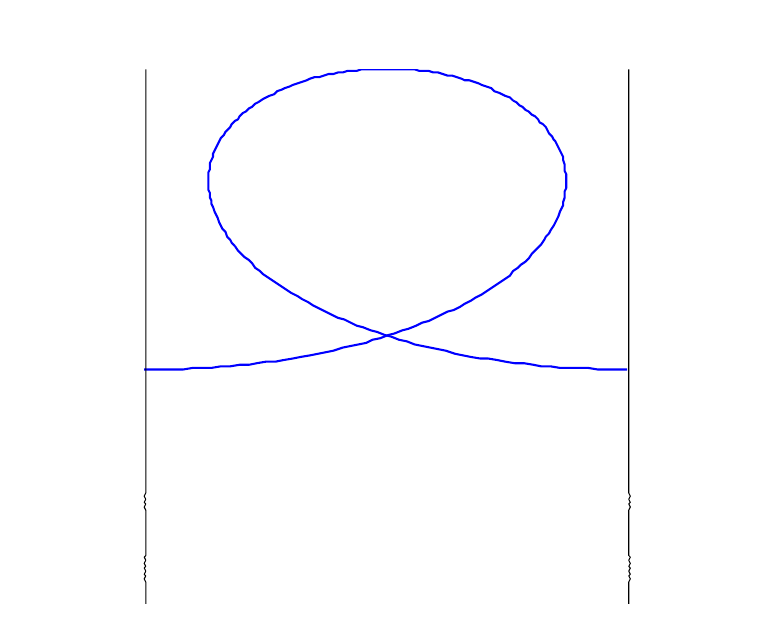}
 &
\\
 &
&
\end{tabular}
\caption{Sample initial data. This initial data has winding number 1.}
\label{Fig1}
\end{figure}

For closed curves, if the oscillation of curvature is initially small, then the
flow exists for all time and converges exponentially fast to a standard circle.
This is the main result of \cite{W13} (see also \cite{GWconvcdf} and an
extension to this in \cite{miuraokabe}). Also in \cite{W13} is an estimate of
the \emph{waiting time}: as the limit is a circle and convergence is smooth,
there exists a $T^*$ such that $k > 0$ for all $t > T^*$, that is, the flow is
eventually convex.
This is interesting in light of \cite{GI}, that shows convexity is in general
lost under the flow.

This is a symptom of the failure of the maximum principle for fourth-order
differential operators.
Another such symptom was identified by Elliott and Maier-Paape \cite{EM01},
that graphicality is typically lost in finite time.
In our situation here, a natural `graph direction' exists: the rotation of $e$ by $\frac\pi2$.
Let us denote this rotated vector by $e^\perp$.
Indeed, analogously to the situation in \cite{W13}, there exists a waiting time
$T^*$ such that for all $t>T^*$, we have
\[
f[\gamma](x,t) := \IP{\nu(x)}{e^\perp} > 0\,,\quad\text{ for all $x\in(-1,1)$}\,.
\]
That is, the flow is eventually graphical.
This leads us to the natural question:

\begin{QN}
Suppose $|e|>0$.
Let $\gamma:(-1,1)\times[0,\infty)\rightarrow\R^2$ be a solution to (CD)
satisfying the assumptions of Theorem \ref{TMmainb}.
Does there exist a $C = C(\gamma(\cdot,0))$ depending only on the initial data such that
\[
\SL\big\{ t\in[0,\infty) : f[\gamma](\cdot,t) \not> 0\big\}
 \le C(\gamma(\cdot,0))
\]
and for every $\varepsilon > 0$ there exists a flow $\gamma_\varepsilon$ such that
\[
\SL\big\{ t\in[0,\infty) : f[\gamma](\cdot,t) \not> 0\big\}
 > C(\gamma_\varepsilon(\cdot,0))-\varepsilon\,?
\]
\end{QN}

In the above we have used $\SL$ to denote Lebesgue measure.

\begin{rmk}
\label{RMKfinitetime}
Finite-time singularities for the curve diffusion flow with closed data remain
difficult to penetrate.  Although there are natural Lyapunov functionals for
the flow, these do not seem to yield classification results for blowups of
singularities.  Indeed, it is still unknown if solutions in symmetric
perturbation classes near non-trivial shrinkers (such as the lemniscate of Bernoulli
discussed in \cite{edwards14}) converge modulo rescaling to the shrinker.
We note that a partial result in this direction is \cite{miuraokabe}.
As mentioned, we can also understand this self-similar solution in the free
boundary setting (see Figure \ref{FigLemniscate}).
It seems likely that the free boundary setting will be useful when studying
perturbations of the lemniscate.

In the free boundary setting, finite-time singularities are more common, and
global analysis of the flow can be quite problematic even in a small data
regime.
For example, the solution space of the exterior problem, where the flow is supported on parallel
lines but with winding number $\omega \ne 0$ (see Figure \ref{FigExterior}), contains no curves with constant curvature.
There is no equilibrium that satisfyies the boundary conditions.
Nevertheless, by adjusting the aperture width $|e|$, it is simple to see that
one may make the oscillation of curvature arbitrarily small.

There is an interesting technical point here.
Some of the estimates used to prove Theorem \ref{TMmainb} are close to optimal:
using initial smallness of the oscillation of curvature, we may use the method
of proof from Lemma \ref{KoESTcaseb} to find that curvature is
well-controlled in $L^2$ if we can control the length difference $L(t) -
L(0)$.
If the supporting lines are skew, this follows by using an isoperimetric-type
argument.  For parallel lines this doesn't work.
If $\omega = 0$ then the problematic term is absent, however for $\omega\ne0$,
the term needs to be estimated.  An easy condition controlling this term is
that $L(0) = |e|+\delta$, where $|e|$ is the length of the straight line
connecting each of the parallel lines.
If it were possible to choose $\delta < K_0$, where $K_0$ is larger than the
initial oscillation of curvature and smaller than $K^*$ from Lemma \ref{KoESTcaseb},
then a stability result would follow.
These requirements are in competition with one another: although the
oscillation of curvature is scale-invariant, decreasing $\delta$ beyond a
certain critical level necessitates an increase in the oscillation of
curvature.  Indeed, the fact that there is no equilibrium in the class of
curves satisfying the boundary conditions for $\omega \ne 0$ proves that it is
not possible to make this choice.  As a corollary of this, we conclude the
following lower bound for the oscillation of curvature in the exterior
problem.
\end{rmk}

\begin{figure}
\begin{tabular}{ccc}
\hspace{-1cm}\includegraphics[trim=5cm 10cm 4cm 9cm,clip=true,width=5cm]{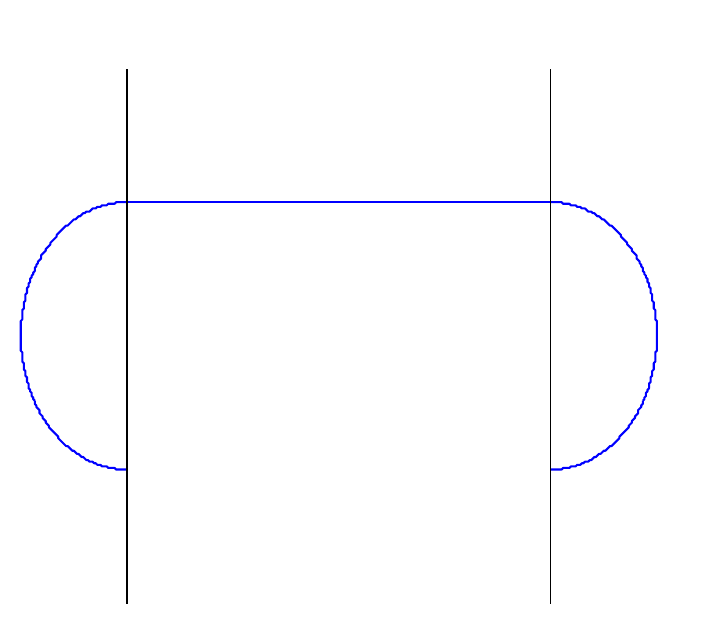}
 &
\includegraphics[trim=5cm 10cm 3cm 9cm,clip=true,width=5cm]{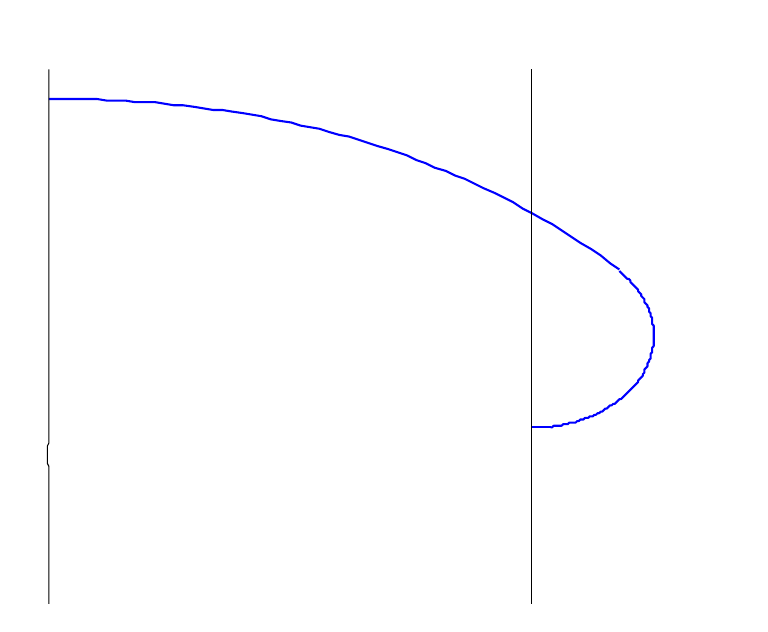}
\\
\hskip-1.2cm
(a) $\omega = 1$
 &
\hskip-0.2cm
(b) $\omega = \frac12$;
\end{tabular}
\caption{Sample initial data for the exterior problem.}
\label{FigExterior}
\end{figure}

\begin{cor}
\label{CYlowerbound}
Suppose $|e|>0$.
Let $\gamma:(-1,1)\rightarrow\R^2$ be an immersed curve satisfying the boundary
conditions of the exterior problem: $\eta_i:\R\rightarrow\R^2$ are parallel
straight lines, the origin lies in the interior of
$\eta_i$\footnote{The interior is the region between the two parallel lines.},
$\gamma$ meets $\eta_i$ at right angles, with $k_s(\pm1) = 0$, and at least one
of the tangent vectors at
the boundary $\tau_i$ points away from the interior of $\eta_i$.

Then
\[
\Ko(\gamma) + 8\pi^2\log\bigg(\frac{L(\gamma)}{|e|}\bigg)
 \ge \frac{12\pi^2\omega^2 + \pi - 2\omega\pi\sqrt{6\pi(6\pi\omega^2+1)}}{3}
\,.
\]
\end{cor}

Concerning the global behaviour of this flow, we make the following conjecture.

\begin{conj}
Suppose $|e|>0$.
There exists an immersed curve $\gamma_0:(-1,1)\rightarrow\R^2$ satisfying the
boundary conditions of the exterior problem, as in Corollary
\ref{CYlowerbound}, with the following property.  The curve diffusion
flow $\gamma:(-1,1)\times[0,T)\rightarrow\R^2$ with
$\gamma_0$ as initial data exists for at most finite time, and
$\gamma(\cdot,t)$ approaches a multiply-covered straight line
in the Hausdorff metric but not in $C^k$ for any $k\geq 1$.
\end{conj}

\subsection{Elastic flow}

We finish by giving a surprising global result on the free elastic flow.
Due to the powerful a-priori estimate on the energy for the flow, global existence follows without too much difficulty.

As noted earlier however, compactness is not expected in general
due to the norm $\vn{\gamma}_\infty$ possibly growing without bound.
In order to obtain compactness, a restriction on length (either penalisation or
constraint) is usually imposed (see for example \cite{DKS02} for the first
result of this kind, and \cite{okabepozziwheeler} and \cite{okabewheeler} for
recent low-regularity developments).

This makes global results on the free elastic flow quite rare.
For the flow supported on parallel lines, we are able to obtain not only global existence, but in fact convergence under a curvature condition.
We do not need to rescale the flow, nor adjust for translations.
We also do not require reparametrisation.
Our result holds for data with initial scale-invariant $L^2$-norm squared of the curvature smaller than $\pi$.
As noted earlier, this can be thought of as a geometric stability result for
the \emph{entire} free elastic flow, as the boundary condition can, via
reflection, be understood as imposing a
\emph{cocompactness condition} on the flow.

\begin{thm}
\label{TMmainbElastic}
Suppose $|e|>0$.
Let $\gamma:(-1,1)\times[0,T)\rightarrow\R^2$ be a solution to (FE).
Assume that
\begin{equation}
\label{EQvanillaEhypothesis1}
L(0)\int_\gamma k^2\,ds\bigg|_{t=0} \le \pi\,.
\end{equation}
Then the flow exists globally $T=\infty$, and $\gamma(\cdot,t)$ converges
exponentially fast to a straight line segment parallel to $e$ in the $C^\infty$
topology.
\end{thm}

\begin{rmk}
As with (CD), it is unknown how to determine, from the initial data, which
straight line the flow will converge to.
\end{rmk}

Sharpness of the given condition is unknown, however, we do not expect it to be
sharp.
Based on the winding number calculation in Lemma \ref{LMlengthestimate} and
numerical evidence, we make the following conjecture.

\begin{conj}
	Theorem \ref{TMmainbElastic} holds with \eqref{EQvanillaEhypothesis1} replaced by
\begin{equation}
\label{EQconj}
L(0)\int_\gamma k^2\,ds\bigg|_{t=0} \le \pi^2\,.
\end{equation}
\end{conj}

The argument for including equality in \eqref{EQconj} above is as follows.
It is possible to construct, for any $\delta > 0$, a smooth curve satisfying the boundary conditions with
\[
	\omega = \frac12\quad\text{and}\quad L\int_\gamma k^2\,ds = \pi^2 + \delta\,.
\]
Clearly such curves can not smoothly converge to a straight line; in fact,
numerical evidence suggests that (unlike (CD) flow) such curves expand
indefinitely and do not display any compactness property.
In particular, length is no longer controlled a-priori.
However, these non-compact examples use $\delta>0$ in an essential way.
As we are not aware of any other obstacles to the result, we conjecture that
\eqref{EQconj} is optimal.

As a final remark, we note that if we enforce $\omega=0$ then the conjectured condition \eqref{EQconj} should be replaced by
\begin{equation}
\label{EQhelp7}
	L(0)\int_\gamma k^2\,ds\bigg|_{t=0} < 4\pi^2\,.
\end{equation}
This is because in the absence of obstacles from the topology (as with the
winding number above), the main source of remaining difficulty lies in the set
of equilibria itself.
The least energy elastica (excluding straight line segments parallel to $e$)
that satisfy the boundary conditions are the
rectangular elastica; these all have $L||k||_2^2 \ge 4\pi^2$.
Although $L||k||_2^2$ is not automatically a preserved quantity for the free
elastic flow, this is good evidence to suggest that the optimal condition is
\eqref{EQhelp7} above.

\section*{Acknowledgements}

The authors thank their colleagues for several useful discussions, in
particular Ben Andrews, James McCoy and Philip Schrader.
This work was completed with each of the authors supported by ARC discovery
project DP150100375 and DP228371030.
Part of this work was completed during a visit by the first author to
Universit\"at Regensburg in July 2012.  At the time he was supported by the
Alexander-von-Humboldt foundation (ID 1137814).  He would like to thank Harald
Garcke, Daniel Depner, and Lars M\"uller for their hospitality and support.

\section{Evolution equations and standard inequalities}

Changes in the Euclidean geometry of the evolving curves can be understood via
the noncommutativity of the Euclidean arc-length and time derivatives.  In this
section and throughout the rest of the paper we reparametrise the evolving
family of curves by arc-length, with arc-length parameter $s$.

\begin{lem}
\label{LMevol1}
Let $\gamma:(-1,1)\times[0,T)\rightarrow\R^2$ be a solution to \eqref{CD} given by Theorem \ref{TMste}.
Then
\[
[\partial_t,\partial_s] = \partial_t\partial_s - \partial_s\partial_t = -kF\partial_s\,.
\]
\end{lem}
\begin{proof}
We compute
\begin{align*}
\partial_t |\gamma_x|^2 &= 2|\gamma_x|\,|\gamma_x|_t
\\
 &= 2\IP{\gamma_x}{\gamma_{tx}}
 = 2\IP{\gamma_x}{(-F\nu)_x}
\\
 &= -2F\IP{\gamma_x}{\nu_x}
\\
 &= 2F\IP{\gamma_x}{\gamma_x}
\\
 &= 2\big(kF\big)|\gamma_x|^2
\,;
\end{align*}
so
\begin{equation}
\label{EQevolareaelement}
\partial_t|\gamma_x|
 = \big(kF\big)|\gamma_x|
\end{equation}
and
\[
\partial_t\partial_s - \partial_s\partial_t
 = \partial_t\Big(\frac{1}{|\gamma_x|}\Big)\partial_x
 = \frac{-kF}{|\gamma_x|}\partial_x
 = -kF\partial_s\,.
\]
\end{proof}
As we shall see below, Lemma \ref{LMevol1} facilitates quick calculation of the
evolution of the tangent, normal, curvature, and derivative of curvature
vectors.
\begin{lem}
\label{LMevo}
Let $\gamma:(-1,1)\times[0,T)\rightarrow\R^2$ be a solution to \eqref{CD} given by Theorem \ref{TMste}.
The following evolution equations hold:
\begin{align*}
\tau_t &= -F_s\nu\,,
\\
\nu_t &= F_s\tau\,,
\\
k_t &= - F_{ss} - Fk^2\,,
\\
k_{st}
&= - F_{s^3} - F_sk^2 - 3Fk_sk
\,.
\end{align*}
\end{lem}
\begin{proof}
Lemma \ref{LMevol1} is used repeatedly in this proof.
Let us begin with
\begin{align*}
\tau_t &= \gamma_{st} = \gamma_{ts} - kF\gamma_s
\\
       &= -(F\nu)_s - kF\tau
\\
       &= -F_s\nu + (kF - kF)\tau
\\
       &= -F_s\nu
\,.
\end{align*}
Since $\IP{\tau}{\nu}=0$, this implies directly
\[
\nu_t = -\IP{\nu}{-F_s\nu}\tau
 = F_s\tau
\,.
\]
Similarly,
\begin{align*}
\kappa_t &= \gamma_{sst} = \gamma_{sts} - kF\gamma_{ss}
\\
 &= -(F_s\nu)_s - kF\kappa
\\
 &= -(F_{ss} + k^2F)\nu + kF_s\tau\,.
\end{align*}
The tangential movement of the curvature vector here can be understood as rotation, whereas the normal velocity is a dilation.
Indeed, the scalar curvature $k = \IP{\nu}{\kappa}$ evolves by
\begin{align*}
k_t &=
  \IP{F_{ss}\tau}{\kappa} +
  \IP{-(F_{ss} + k^2F)\nu + (kF_s)\tau}{\nu}
\\
&= - F_{ss} - Fk^2
\,.
\end{align*}
For $\kappa_s = k_s\nu - k^2\tau$ we proceed as before:
\begin{align*}
\kappa_{st} &= \kappa_{ts} - kF\kappa_s
\\
 &=  \big[(-F_{ss} - k^2F)\nu + kF_s\tau\big]_s - kF\kappa_s
\\
 &=  \big(-F_{s^3} - 2kk_sF - k^2F_s + k^2F_s - kk_sF\big)\nu
    + X\tau
\\
 &=  \big(-F_{s^3} - 3kk_sF\big)\nu
    + X\tau
\,.
\end{align*}
In the above $X = (kF_s)_s + kF_{ss}$, but this isn't important as this term will vanish when we take an inner product with $\nu$.
That is what we do now: Since $k_s = \IP{\kappa_s}{\nu}$, this implies
\begin{align*}
k_{st} &=
  \IP{F_s\tau}{k_s\nu-k^2\tau}
 +
  \IP{
     \big(-F_{s^3} - 3kk_sF\big)\nu
    + X\tau}{\nu}
\\&= - F_{s^3} - F_sk^2 - 3Fk_sk
\,.
\end{align*}
\end{proof}
\begin{rmk}
We note that the above evolution equation can be obtained without requiring explicit calculation of $\kappa_{st}$ by working directly
with the scalar quantites:
\begin{align*}
k_{st} &= k_{ts} - Fk_sk
\\
       &= \big(-F_{ss} - Fk^2\big)_s - Fk_sk
\\
&= - F_{s^3} - F_sk^2 - 3Fk_sk
\,.
\end{align*}
We included the method using the evolution of the vector $\kappa_s$ as we feel the quantity $\kappa_{st}$, and in particular the
nontrivial cancellations involved in moving from $\kappa_s \mapsto k_s$ are also of interest.
\end{rmk}

Induction on the commutator relations yields the following generic formula for the evolution of the $l$-th
derivative of curvature.
Similar formulae were derived in \cite[Lemma 2.3]{DKS02}.

\begin{lem}
\label{LMevo2}
Let $\gamma:(-1,1)\times[0,T)\rightarrow\R^2$ be a solution to \eqref{CD} given by Theorem \ref{TMste}.
The evolution of the $l$-th derivative of curvature 
\begin{itemize}
	\item[(CD)] along the curve diffusion flow is given by
\[
\partial_tk_{s^l}
 = -k_{s^{(l+4)}} + \sum_{q+r+u=l} c_{qru}k_{s^{(q+2)}}k_{s^r}k_{s^u}\,,
\]
for constants $c_{qru}\in\R$ with $q,r,u\geq 0$;
\item[(FE)] along the elastic flow is given by
\[
\partial_tk_{s^l}
 = -k_{s^{(l+4)}} + \sum_{q+r+u=l} c_{qru}k_{s^{(q+2)}}k_{s^r}k_{s^u}
 + \sum_{q+r+u+v+w=l} c_{qruvw}k_{s^q}k_{s^r}k_{s^u}k_{s^v}k_{s^w}\,,
\]
for constants $c_{qru}, c_{qruvw}\in\R$ with $q,r,u,v,w\geq 0$.
\end{itemize}
\end{lem}

Equation \eqref{EQevolareaelement} gives the evolution of the length element, and thus, of the change in
length of the evolving curves.

\begin{lem}
\label{LMevol2}
Let $\gamma:(-1,1)\times[0,T)\rightarrow\R^2$ be a solution to \eqref{CD} given by Theorem \ref{TMste}.
Then
\[
L'(\gamma(\cdot,t)) =   \int_\gamma Fk\,ds\,.
\]
In particular for (CD) flow we have
\[
L'(\gamma(\cdot,t)) = - \int_\gamma k_s^2\,ds\,.
\]
\end{lem}
\begin{proof}
Using \eqref{EQevolareaelement} we compute
\[
\frac{d}{dt}\int_{-1}^1 |\gamma_u|\,du
 = \int_{-1}^1
    Fk\,|\gamma_x|\,dx\,.
\]
If $F = k_{ss}$ then
\[
L'(\gamma(\cdot,t)) 
 = -\int_\gamma k_s^2\,ds\,,
\]
as required. Note that we used the boundary conditions in the last step.
\end{proof}

Note that the elastic flow tends to increases length and for the constrained
elastic flow although we have $L(\gamma(\cdot,t)) \le L(\gamma(\cdot,0))$, we
do not have monotonicity of the length in general.

There exists an $\omega\in\R$ satisfying
\begin{equation}
\int_{\gamma} k\, ds\bigg|_{t=0} = 2\omega\pi.
\label{TC}
\end{equation}
In the case where the solution is a family of closed curves, $\omega$ is the winding number of $\gamma(\cdot,0)$.
In the cases we consider here, we continue to call $\omega$ the winding number, although it is no longer
guaranteed to be an integer.
The lemma below shows that, as expected, it is constant along the flow.
Note that it requires the free boundary condition and the flatness of the support lines $\eta_i$.

\begin{lem}
\label{WN}
Suppose $|e|>0$.
Let $\gamma:(-1,1)\times[0,T)\rightarrow\R^2$ be a solution to \eqref{CD} given by Theorem \ref{TMste}.
Set
\[
\int_{\gamma} k\, ds\bigg|_{t=0} = 2\omega\pi.
\]
Then
\[
\int_{\gamma} k\, ds = 2\omega\pi.
\]
In particular, for (CD) flow the average of the curvature $\kav$ increases in
absolute value with velocity
\[
\rd{}{t}|\kav| = \frac{2|\omega|\pi}{L^2}\vn{k_s}_2^2.
\]
\end{lem}
\begin{proof}
We set the origin to be any point on the line equidistant from the two parallel lines $\eta_i$.
Recall that $e$ is a vector perpendicular to $\eta_i$ and has length equal to the distance between
them.
The Neumann condition is equivalent to
\begin{equation*}
\IP{\nu(\pm1,t)}{e} = 0\,.
\end{equation*}
Differentiating this in time yields
\begin{equation*}
F_{s}(\pm1,t)\IP{\tau(\pm1,t)}{e} = \pm |e| \ F_{s}(\pm1,t) = 0\,.
\end{equation*}
Now $|e| \ne 0$ so we must have that
\begin{equation}
\label{EQkszero}
F_{s}(\pm1,t) = 0\,.
\end{equation}
We compute
\[
\rd{}{t}\int_{\gamma} k\, ds
 = - \int_{\gamma} F_{ss} + Fk^2 - Fk^2\,ds
 = - \int_{\gamma} F_{ss}\, ds
 = - F_{s}|_{\{-1,1\}}
 = 0\,.
\]
This completes the proof.
\end{proof}

Differentiating the boundary conditions can be taken quite far, as the following lemma shows.

\begin{lem}
\label{LMbdy}
Let $\gamma:(-1,1)\times[0,T)\rightarrow\R^2$ be a solution to \eqref{CD} given by Theorem \ref{TMste}.
For all $p\in\N$
\[
k_{s^{(2p+1)}}\Big|_{\{0,L\}} = 0\,.
\]
\end{lem}
\begin{proof}
Differentiating the no-flux condition $k_s(\pm1,t) = 0$ we find
\begin{equation}
\label{EQnoflux1}
(
 - F_{s^3} - F_sk^2 - 3Fk_sk
)(\pm1,t) = 0\,.
\end{equation}
Substituting the no-flux condition and \eqref{EQkszero} into \eqref{EQnoflux1} we find
\begin{equation}
\label{EQnew1}
F_{s^3}(\pm1,t) = 0\,.
\end{equation}
For the \eqref{CD} flows these imply together with an induction argument that
\begin{equation}
	\label{EQbdy1}
	k_{s^{(2p-1)}}(\pm1,t) = 0
\end{equation}
for $p = 1,2,3$.
For clarity we calculate in each case separately.

{\bf (CD) flow.}
For curve diffusion flow the claim \eqref{EQbdy1} is immediate.

{\bf (FE) flow.}
For the elastic flow we have at the boundary
\[
	0 = F_s = k_{s^3} + \frac32k^2k_s = k_{s^3}
\]
and
\begin{align*}
	0 &= F_{s^3} = k_{s^5} + \frac32(k^2k_s)_{ss}
	\\
	&= k_{s^5} + \frac32(2kk_s^2 + k^2k_{ss})_{s}
	\\
	&= k_{s^5} + \frac32(2k_s^3 + 6kk_sk_{ss} + k^2k_{s^3})
	\\
	&= k_{s^5}\,.
\end{align*}
We conclude again \eqref{EQbdy1}.

Let us give the induction argument.
We assume that for all $p\in{1,\ldots,n}$,
\begin{equation}
\label{EQinduc}
k_{s^{(2p-1)}}(\pm1,t) = 0\,.
\end{equation}
The evolution of the $l$-th derivative of curvature is given by
\begin{align*}
\partial_tk_{s^l}
&= -k_{s^{(l+4)}} + \sum_{q+r+u=l} c_{qru}k_{s^{(q+2)}}k_{s^r}k_{s^u}
 + \sum_{q+r+u+v+w=l} c_{qruvw}k_{s^q}k_{s^r}k_{s^u}k_{s^v}k_{s^w}
 \,,
 \end{align*}
for constants $\hat c_{qru}, c_{qru}, c_{qruvw}\in\R$ with $q,r,u,v,w\geq 0$.
Note that for curve diffusion flow $c_{qruvw} = 0$.

The inductive hypothesis implies that, for $l$ odd and less than or equal to $2n-1$, the
derivative $k_{s^l}$ vanishes on the boundary.
Let's take $l = 2n-3$.
Then we have (evaluating at the boundary)
\begin{align*}
k_{s^{(2n+1)}}
&= - \sum_{2q_1 + 2r_1 + 2u_1 = 2n-3}k_{s^{(2q_1+2)}}*k_{s^{(2r_1)}}*k_{s^{(2u_1)}}
\\
&\quad - \sum_{2q_2+2r_2+2u_2+2v_2+2w_2=2n-3} k_{s^{(2q_2)}}*k_{s^{(2r_2)}}*k_{s^{(2u_2)}}*k_{s^{(2v_2)}}*k_{s^{(2w_2)}}
\,.
\end{align*}
In the above equation we have used $*$ to denote a linear combination of terms with absolute coefficient.
Using the hypothesis \eqref{EQinduc}, we have removed all terms with an odd
number of derivatives of $k$.
Each sum is therefore taken over all $q_i, r_i, u_i, v_i, w_i$ such that $2q_i + 2r_i + 2u_i = 2p-3$ or
$2q_i+2r_i+2u_i+2v_i+2w_i = 2p-3$, which is the empty set.
We conclude
\[
k_{s^{(2n+1)}} = 0\,,
\]
equivalent to \eqref{EQbdy1}, as required.
\end{proof}

We will need the following Sobolev-Poincar\'e-Wirtinger inequalities.
The statements below are given in the arc-length parametrisation, but we note that they continue to hold under any
change of measure, with the appropriate change to $L$.

\begin{lem}
\label{WIp}
Suppose $f:[0,L]\rightarrow\R$, $L>0$, is absolutely continuous and $\int_0^L f\,ds = 0$.
Then
\[
\int_0^L f^2\, ds \le \frac{L^2}{\pi^2}\int_0^L f_s^2 ds\,.
\]
\end{lem}
\begin{proof}
Form the new function $g:[-L,L]\rightarrow\R$ defined by $g(x) = f(x)$ for $x>0$ and $g(x) = f(-x)$ for $x\le 0$.
Then $g$ is periodic, since $g(-L) = g(L)$, and furthermore
\[
	\int_{-L}^L g\,ds = 2\int_0^L f\,ds = 0\,.
\]
Applying the standard Poincar\'e inequality to $g$ we find
\[
	\int_{-L}^L g^2\,ds \le \frac{L^2}{\pi^2} \int_{-L}^L g_s^2\,ds\,.
\]
This implies
\[
	2\int_{0}^L f^2\,ds \le \frac{2L^2}{\pi^2} \int_{0}^L f_s^2\,ds\,,
\]
as required.
\end{proof}

\begin{lem}
\label{WI}
Suppose $f:[0,L]\rightarrow\R$, $L>0$, is absolutely continuous and $f(0) = f(L) = 0$.
Then
\[
\int_0^L f^2\, ds \le \frac{L^2}{\pi^2}\int_0^L f_s^2 ds\,.
\]
\end{lem}
\begin{proof}
Define $g(s) = f(s)$ for $s\in[0,L]$ and $g(s) = -f(-s)$ for $s\in[-L,0)$. Then $g$ is a continuous odd periodic function.
Since $\int g\,ds = 0$, the standard Poincar\'e inequality implies
\[
2\int_0^L f^2\,ds = \int_{-L}^L g^2\,ds
 \le \frac{4L^2}{4\pi^2} \int_{-L}^L g_s^2\,ds
 = \frac{2L^2}{\pi^2} \int_{0}^L f_s^2\,ds
\,.
\]
\end{proof}

\begin{cor}
\label{LU}
Under the assumptions of Lemma \ref{WI}
\[
\vn{f}_\infty^2 \le \frac{L}{\pi}\vn{f'}_2^2.
\]
Under the assumptions of Lemma \ref{WIp}
\[
\vn{f}_\infty^2 \le \frac{2L}{\pi}\vn{f'}_2^2.
\]
\end{cor}
\begin{proof}
First assume $f:[0,L]\rightarrow\R$, $L>0$, is absolutely continuous and $f(0)
= f(L) = 0$.
Periodicity implies
\[
f^2(s) = \int_{0}^s f\,f'\,\,ds - \int_s^{L} f\,f'\,ds\,.
\]
Therefore
\[
f^2(s) \le \int_0^L |f\,f'|\,ds\,.
\]
Now H\"older's inequality and Lemma \ref{WI} above implies
\[
\vn{f}_\infty^2 \le \vn{f}_2\vn{f'}_2 \le \frac{L}{\pi}\vn{f'}_2^2,
\]
as required.

It remains to prove the lemma in the case where $f:[0,L]\rightarrow\R$, $L>0$,
is absolutely continuous and $\int_0^L f\,ds = 0$.
In this case there exists a point $p\in[0,L]$ where $f(p) = 0$.
Absolute continuity implies
\[
f^2(s) = 2\int_p^s f\,f'\,ds
\,.
\]
Therefore
\[
\vn{f}_\infty^2 \le 2\vn{f}_2\vn{f'}_2 \le \frac{2L}{\pi}\vn{f'}_2^2\,.
\]
\end{proof}

We now compute the evolution of $\Ko$ for (CD) flow.

\begin{lem}
\label{KoE}
Let $\gamma:(-1,1)\times[0,T)\rightarrow\R^2$ be a solution to (CD) given by Theorem \ref{TMste}.
Then
\begin{align*}
\rd{}{t}\Ko &+ \Ko\frac{\vn{k_s}_2^2}{L} + 2L\vn{k_{ss}}_2^2
\\
 &= 3L\int_{\gamma} (k-\kav)^2k^2_s ds
   + 6\kav L \int_{\gamma} (k-\kav)k_s^2 ds
   + 2\kav^2 L \vn{k_s}^2_2.
\end{align*}
\end{lem}
\begin{proof}
First, note that
\begin{align*}
   2L\int_{\gamma} (k-\kav)\big( -k_{s^4} \big)ds
 = -2L\vn{k_{ss}}_2^2
 + 2L[k_sk_{ss}]_{\{-1,1\}}
 - 2L[(k-\kav)k_{s^3}]_{\{-1,1\}}\,.
\end{align*}
From \eqref{EQkszero} in the proof of Lemma \ref{WN}, the last term vanishes.
The no-flux boundary condition means that the second term vanishes.
Therefore,
\[
   2L\int_{\gamma} (k-\kav)\big( -k_{s^4} \big)ds
 = -2L\vn{k_{ss}}_2^2\,.
\]
The no-flux boundary condition also means that
\[
     2L\int_{\gamma} (k-\kav)\big( -k^2k_{ss}
                                  \big)ds
 = 2L\int_\gamma k^2k_s^2\,ds
  + 4L\int_\gamma (k-\kav)kk_s^2\,ds
\]
and
\[
     L\int_{\gamma} kk_{ss}(k-\kav)^2\,ds
 = - L\int_\gamma (k-\kav)^2k_s^2\,ds
   - 2L\int_\gamma k(k-\kav)k_s^2\,ds
\,.
\]
Using these identities, we compute
\begin{align*}
\rd{}{t}\Ko
 &= - \int_{\gamma}k_s^2 ds \int_{\gamma} (k-\kav)^2 ds
   + 2L\int_{\gamma} (k-\kav)\big( -k_{s^4}-k^2k_{ss}
                                  \big)ds
\\
 &\qquad  + L\int_{\gamma} kk_{ss}(k-\kav)^2 ds
\\
 &= - \Ko\frac{\vn{k_s}_2^2}{L} - 2L\vn{k_{ss}}_2^2
    + 2L\int_{\gamma} k^2k_s^2 ds
    + 4L\int_{\gamma} k(k-\kav)k^2_{s} ds
\\
 &\qquad
    - L\int_{\gamma} (k-\kav)^2k_s^2 ds
    - 2L\int_{\gamma} k(k-\kav)k^2_s ds
\\
 &= - \Ko\frac{\vn{k_s}_2^2}{L} - 2L\vn{k_{ss}}_2^2
    + 4L\int_{\gamma} k^2k_s^2 ds
    - 2\kav L\int_{\gamma} kk^2_{s} ds
\\
 &\qquad
    - L\int_{\gamma} (k-\kav)^2k_s^2 ds.
\end{align*}
Rearranging, we have
\begin{align*}
\rd{}{t}\Ko &+ \Ko\frac{\vn{k_s}_2^2}{L} + 2L\vn{k_{ss}}_2^2
\\
 &=   4L\int_{\gamma} k^2k_s^2 ds
    - 2\kav L\int_{\gamma} (k-\kav)k^2_{s} ds
    - 2\kav^2 L\int_{\gamma} k^2_{s} ds
    - L\int_{\gamma} (k-\kav)^2k_s^2 ds
\\
 &=   3L\int_{\gamma} (k-\kav)^2k_s^2 ds
    + 6\kav L\int_{\gamma} (k-\kav)k^2_{s} ds
    + 2\kav^2 L\int_{\gamma} k^2_{s} ds.
\end{align*}
This proves the lemma.
\end{proof}

\section{Curvature estimates in $L^2$ and global analysis}

In this section we describe how the evolution equations can be used to prove
a-priori curvature estimates, and then how these estimates imply global
existence and convergence.
A fundamental observation is that for the curve diffusion flow Lemma
\ref{LMevol2} and Lemma \ref{WIp} together imply $\Ko\in L^1([0,T))$.

\begin{lem}
\label{FO}
Let $\gamma:(-1,1)\times[0,T)\rightarrow\R^2$ be a solution to (CD) given by
Theorem \ref{TMste}.
Then $\vn{\Ko}_1 < L^4(0)/4\pi^2$.
\end{lem}
\begin{proof}
Applying Lemma \ref{WIp} with $f=k-\kav$ and recalling Lemma \ref{LMevol2} we have
\[
\Ko \le \frac{L^3}{\pi^2}\vn{k_s}_2^2 = - \frac{1}{4\pi^2}\rd{}{t}L^4,
\]
so
\[
\int_0^t\Ko\, d\tau \le \frac{L^4(0)}{4\pi^2} - \frac{L^4(t)}{4\pi^2}.
\qedhere
\]
\end{proof}

The above lemma holds regardless of initial data, in particular regardless of
the winding number, indicating that the quantity $\Ko$ is a natural energy for
the curve diffusion flow with free boundary, just as it is for the curve
diffusion flow of closed curves \cite{W13}.

\begin{rmk}
A similar argument as above also shows that $\vn{k_s}^2_2 \in L^1([0,T))$ with the estimate
$\vn{\vn{k_s}_2^2}_1 \le L(0)$, suggesting that $\vn{k_s}^2_2$ is another well-behaved quantity under the
flow.
\end{rmk}


We begin our study of the free elastic flow by proving that there exists a
critical energy level $\pi$ such that while the oscillation of curvature is below
$\pi$ length is monotone.

\begin{lem}
\label{LMlengthestimate}
Let $\gamma:(-1,1)\times[0,T)\rightarrow\R^2$ be a solution to (FE)
given by Theorem \ref{TMste}.
Assume that for all $t\in[0,T)$
\begin{equation}
\label{EQvanillaEhypothesis}
L\int_\gamma k^2\,ds \le \pi\,.
\end{equation}
Then $\omega = 0$ and
\[
	L(\gamma(\cdot,t)) \le L(\gamma(\cdot,0))\,,\quad\text{for all}\quad t\in[0,T)\,.
\]
\end{lem}
\begin{proof}
Let us first show that $\omega = 0$.
Note that Lemma \ref{WN} implies that $\omega$ is constant along the flow, and so it suffices to show that $\omega = 0$ at initial time.
A calculation analogous to that in Remark \ref{RMKomegazero} shows that
\[
\omega \le \frac{1}{2\sqrt{\pi}}
\,.
\]
As $\omega$ is an integer multiple of $\frac12$, it must be zero.

Therefore by Lemma \ref{LU}, we have
\[
	\vn{k}_\infty^2 \le \frac{2L}{\pi}\vn{k_s}_2^2
	\,.
\]
Using this we calculate
\begin{align*}
	L'(\gamma(\cdot,t)) &= \int_\gamma kk_{ss} + \frac12k^4 \,ds
	\\
	&= -\vn{k_s}_2^2 + \frac12\vn{k}_4^4
	\\
	&\le -\vn{k_s}_2^2 + \frac12\vn{k}_\infty^2\vn{k}_2^2
	\\
	&\le -\vn{k_s}_2^2\Big(1-\frac{L}{\pi}\vn{k}_2^2\Big)
	\\
	&\le 0\,.
\end{align*}
The claim follows.
\end{proof}

We note that the definition of the elastic flow implies $\vn{k}_2^2$ is
monotone decreasing, and so with Lemma \ref{LMlengthestimate} the product is
also decreasing. This product, in the setting of $\omega = 0$, is $\Ko$, the
oscillation of curvature.

\begin{lem}
	\label{LMkoscestimate}
Let $\gamma:(-1,1)\times[0,T)\rightarrow\R^2$ be a solution to (FE)
given by Theorem \ref{TMste}.
Assume that for $K\in\R$ we have
\begin{equation}
\label{EQvanillaEhypothesis00}
L(0)\int_\gamma k^2\,ds\bigg|_{t=0} \le K \le \pi\,.
\end{equation}
Then for all $t\in[0,T)$
\[
	\Ko(\gamma(\cdot,t)) \le \Ko(0) = K\,.
\]
\end{lem}
\begin{proof}
By the calculation in the proof of Lemma \ref{LMlengthestimate} we have
\[
	L'(0) \le 0\,.
\]
We calculate
\begin{align*}
	\Big(\frac{d}{dt}\Ko\Big)(\gamma(\cdot,0))
	&=  	L'(0)\int_{\gamma_0} k^2\,ds
	- 2L(0)\int_{\gamma_0} \Big|k_{ss} + \frac12k^3\Big|^2\,ds
	\le 0\,.
\end{align*}
Therefore both $L$ and $\Ko$ are decreasing while $\Ko \le K \le \pi$. Since
this is true at initial time, it remains true, and we conclude the estimate.
\end{proof}

In fact we can obtain much more, bounding also $\vn{k_s}_2^2$ a-priori.

\begin{lem}
	\label{LMdsdecay0}
Let $\gamma:(-1,1)\times[0,T)\rightarrow\R^2$ be a solution to (FE)
given by Theorem \ref{TMste}.
Assume that
\begin{equation}
\label{EQvanillaEhypothesis0}
L(0)\int_\gamma k^2\,ds\bigg|_{t=0} \le \frac{4\pi}{7}\,.
\end{equation}
Then for all $t\in[0,T)$
\[
	\int_\gamma k_s^2\,ds
	\le \frac{3L_0^2K_1}{K_1t + 3L_0^2}
\]
where $L_0 = L(\gamma(\cdot,0))$ and
\[
	K_1 = \int_\gamma k_s^2\,ds\bigg|_{t=0}\,.
\]
\end{lem}
\begin{proof}
	We calculate
	\begin{align*}
		\frac{d}{dt}\int_\gamma k_s^2\,ds
		&= \int_\gamma k_s(-2F_{s^3} - 2F_sk^2 - 6Fk_sk + k_skF)\,ds
		\\
		&= \int_\gamma -2k_{s^3}F_{s} - 2k_sk^2F_s - 5k_s^2kF \,ds
		\\
		&= -\int_\gamma k_{s^3}(2k_{s^3} + 3k^2k_s) \,ds
		\\&\qquad
		+ \int_\gamma (k_sk^2)_s(2k_{ss} + k^3)\,ds
		\\&\qquad
		- \frac52\int_\gamma k_s^2k(2k_{ss} + k^3)\,ds
		\\
		&= - 2\vn{k_{s^3}}_2^2
		+ 3\int_\gamma k_{ss}(k^2k_{ss} + 2kk_s^2) \,ds
		\\&\qquad
		- 2\int_\gamma (k_sk^2)(k_{s^3}) \,ds
		- 3\int_\gamma k_s^2k^4\,ds
		\\&\qquad
		- \frac53\int_\gamma (k_{s}^3)_sk\,ds
		- \frac52\int_\gamma k_s^2k^4\,ds
		\\
		&= - 2\vn{k_{s^3}}_2^2
		+ 3\int_\gamma k^2k_{ss}^2\,ds
		- \frac13\int_\gamma k_s^4 \,ds
		\\*&\qquad
		- 2\int_\gamma (k_sk^2)(k_{s^3}) \,ds
		- \frac{11}2\int_\gamma k_s^2k^4\,ds\,.
	\end{align*}
Applying the estimate $2|(k_sk^2)(k_{s^3})| \le \frac2{11}k_{s^3}^2 + \frac{11}2k_s^2k^4$, and using Corollary 2.9 as well as the H\"older inequality and the length bound we find
	\begin{align*}
		\frac{d}{dt}\int_\gamma k_s^2\,ds
		&\le -\Big(2-\frac2{11}\Big)\vn{k_{s^3}}_2^2
		+ 3\int_\gamma k^2k_{ss}^2\,ds
		- \frac13\int_\gamma k_s^4 \,ds
		\\&\le
		-\frac{19}{11}\vn{k_{s^3}}_2^2
		+ 3\vn{k_{ss}}_\infty^2 \int_\gamma k^2\,ds
		- \frac13\int_\gamma k_s^4 \,ds
		\\&\le
		-\frac{19}{11}\vn{k_{s^3}}_2^2
		+ \frac{6}{\pi}\Ko\vn{k_{s^3}}_2^2
		- \frac{1}{3L}\vn{k_s}_2^4
		\\&\le
		-\bigg(\frac{19}{11} - \frac{6}{\pi}\Ko\bigg)\vn{k_{s^3}}_2^2
		- \frac{1}{3L_0}\vn{k_s}_2^4
		\,.
	\end{align*}
Now observe that if $\Ko \le \varepsilon \le \pi$ initially then this is by Lemma \ref{LMkoscestimate} preserved.
Using this with $\varepsilon = \frac{\pi}{6}\frac{19}{11} = \frac{19\pi}{66}$, we obtain for $\phi = \vn{k_s}_2^2$, $C = \frac{1}{3L_0}$,
\[
	\phi' \le -C\phi^2
\]
which implies
\[
	\phi(t) \le \frac{\phi(0)}{C\phi(0)t+1}
\]
as required.
\end{proof}

Unfortunately, the hypothesis \eqref{EQvanillaEhypothesis0} is stronger than we
would like; it is significantly worse than assuming simply
\eqref{EQvanillaEhypothesis00}, which is what is required for our a-priori
length estimate.

Furthermore, blowup analysis indicates that we may bypass the estimate entirely.
Assuming that $\vn{k_s}_2^2(t_j)\rightarrow\infty$ for a sequence of times
$t_j\rightarrow T < \infty$ we consider a sequence of rescalings $\gamma_j(x,t)
= r_j\gamma(x,t_j+r_j^{-4}t)$.
The sequence $\{r_j\}$ is chosen such that $\vn{k_s}_{2,B_1}^2 = 1$.
By a standard compactness theorem and local estimates for the flow we extract a
limit $\gamma_\infty$ which is an entire elastic flow.
Now this limit has $\vn{k_s}_2^2 \ge 1$ by construction, however by scale
invariance and monotonicity of $\Ko$ it is an elastica.
Elasticae can exist at a variety of energy levels, however if $\Ko$ is finite
for an entire curve, then it can only be zero and we have a straight line,
contradicting the concentration of $\int k_s^2\,ds$.
This is not a proof, primarily because $\Ko$ does not converge to anything
along a blowup sequence and makes no sense on an entire curve as the length is
infinite.
However it does clearly suggest that we can do better than
\eqref{EQvanillaEhypothesis0}.

It is our goal to prove that \eqref{EQvanillaEhypothesis00} is sufficient.
First, the interpolation approach of Dziuk-Kuwert-Sch\"atzle \cite{DKS02} allows us to establish
our first basic a-priori estimate with the weaker smallness assumption \eqref{EQvanillaEhypothesis00}.

\begin{lem}
\label{LMksestimate}
Let $\gamma:(-1,1)\times[0,T)\rightarrow\R^2$ be a solution to (FE)
given by Theorem \ref{TMste}.
Assume that
\begin{equation*}
L(0)\int_\gamma k^2\,ds\bigg|_{t=0} \le \pi\,.
\end{equation*}
Then there exists a universal constant $C\in(0,\infty)$ such that for all $t\in[0,T)$
	\[
		\int_\gamma k_s^2\,ds \le \int_\gamma k_s^2\,ds\bigg|_{t=0} + C
		\,.
	\]
\end{lem}
\begin{proof}
The previous proof established that
\begin{align*}
	\frac{d}{dt}\int_\gamma k_s^2\,ds
	&\le - \frac{19}{11}\vn{k_{s^3}}_2^2
	+ 3\int_\gamma k^2k_{ss}^2\,ds
	- \frac13\int_\gamma k_s^4 \,ds
	\,.
\end{align*}
An interpolation inequality from Dziuk-Kuwert-Sch\"atzle \cite[Proposition 2.5]{DKS02}
(see \cite[Appendix C]{dall2014} for details, which additionally deals with
non-closed curves) implies that there exists a universal constant $C$ such that
\[
3	\int_\gamma k^2k_{ss}^2\,ds
	\le \frac8{11}\int_\gamma k_{s^3}^2\,ds + C\vn{k}_2^{14}
	\,.
\]
Therefore
\begin{align*}
	\frac{d}{dt}\int_\gamma k_s^2\,ds
	&\le - \vn{k_{s^3}}_2^2
	- \frac13\int_\gamma k_s^4 \,ds
	+ C\vn{k}_2^{14}
	\\
	&\le
	- \frac1{3L}\bigg(\int_\gamma k_s^2 \,ds\bigg)^2
	+ C\vn{k}_2^{14}(0)
	\,,
\end{align*}
so that, with $\phi(t) = \vn{k_s}_2^2(t)$
\[
\phi'(t) + c\phi^2(t) \le C\,,
\]
where $C$ is a universal constant.
The result follows.
\end{proof}

We now prove that solutions to (FE) satisfying \eqref{EQvanillaEhypothesis00} exist globally in time.

\begin{prop}
\label{CYglobal}
Let $\gamma:(-1,1)\times[0,T)\rightarrow\R^2$ be a solution to (FE)
given by Theorem 1.1.
Assume that
\begin{equation*}
L(\gamma_0)\int_\gamma k^2\,ds\bigg|_{t=0} \le \pi\,.
\end{equation*}
Then $T = \infty$.
\end{prop}
\begin{proof}
In light of Theorem \ref{TMste} and Lemma \ref{LMksestimate}, it remains only to estimate the position vector.
First, we calculate
\begin{align*}
\frac{d}{dt}\int_\gamma |\gamma|^2\,ds
&= \int_\gamma \IP{\gamma}{2F\nu + kF\gamma}\,ds
\\
&\le \int_\gamma k_{ss}(2\IP{\gamma}{\nu} + k|\gamma|^2)\,ds
 + 2\vn{k}^3_\infty\int_\gamma |\gamma|\,ds
 + \vn{k}^4_\infty\int_\gamma |\gamma|^2\,ds
\,.
\end{align*}
Now Lemma 3.5 and Corollary 2.9 (together with the length bound, Lemma 3.2)
imply that $\vn{k}_\infty(t) \le C$ where $C$ depends only on the initial data
$\gamma_0$.
Therefore, using this as well as $|\gamma| \le 1 + \frac14|\gamma|^2$ and integration by parts, we estimate
\begin{align*}
\frac{d}{dt}\int_\gamma |\gamma|^2\,ds
&\le -\int_\gamma k_s(-2k\IP{\gamma}{\tau} + k_s|\gamma|^2 + 2k\IP{\gamma}{\tau})\,ds
\\
&\qquad + C + C\int_\gamma |\gamma|^2\,ds
\\
&= -\int_\gamma k_s^2|\gamma|^2\,ds
 + C + C\int_\gamma |\gamma|^2\,ds
\,.
\end{align*}
The above estimate is of the form
\[
\phi'(t) \le C + C\phi(t)
\]
with $\phi(t) = \vn{\gamma}_2^2(t)$ and so the Gr\"onwall inequality applies to give
\begin{equation}
\label{EQposvec}
\vn{\gamma}_2^2(t) \le Ce^t
\end{equation}
where $C$ is a constant depending only on $\gamma_0$.
To upgrade this to the required $L^\infty$ estimate we apply Corollary 2.9 to
the components $\gamma^1$, $\gamma^2$ of $\gamma$ in any orthonormal basis as
follows.
In the following calculation, we use $\overline{\gamma^i}$ to denote the average of $\gamma^i$: $\gammaiav = \frac1L\int_\gamma \gamma^i\,ds$.
We calculate
\begin{align*}
\vn{\gamma}_\infty
 &= \vn{\gamma^1e_1 + \gamma^2e_2}_\infty
\\&\le
  \vn{\gamma^1}_\infty + \vn{\gamma^2}_\infty
\\&\le
\vn{\gamma^1 - \overline{\gamma^1} + \overline{\gamma^1}}_\infty
+ \vn{\gamma^2 - \overline{\gamma^2} + \overline{\gamma^2}}_\infty
\\&\le
\vn{\gamma^1 - \overline{\gamma^1}}_\infty
 + \vn{\overline{\gamma^1}}_\infty
+ \vn{\gamma^2 - \overline{\gamma^2}}_\infty
 + \vn{\overline{\gamma^2}}_\infty
\\&\le
	\bigg(
		\frac{2L}{\pi}\int_\gamma |\tau^1|^2 \,ds
	\bigg)^{\frac12}
	+ \bigg(
		\frac{2L}{\pi}\int_\gamma |\tau^2|^2 \,ds
	\bigg)^{\frac12}
+ \frac1L\int_\gamma \gamma^1 + \gamma^2\,ds
\\&\le
	2L
	\bigg(
		\frac{2}{\pi}
	\bigg)^{\frac12}
+ \frac{\sqrt 2}L \int_\gamma |\gamma|\,ds\,,\qquad\qquad\text{ since $|a + b| \le \sqrt2\sqrt{a^2 + b^2}$,}
\\&\le
	2L_0
	\bigg(
		\frac{2}{\pi}
	\bigg)^{\frac12}
+ \frac{\sqrt 2}{L^{\frac12}} \bigg(\int_\gamma |\gamma|^2\,ds\bigg)^{\frac12}
\\&\le
	2L_0
	\bigg(
		\frac{2}{\pi}
	\bigg)^{\frac12}
+ \frac{\sqrt 2}{|e|^{\frac12}} \vn{\gamma}_2
\,.
\end{align*}
We also used the a-priori bounds for length above.

Using now \eqref{EQposvec} to estimate the term $\vn{\gamma}_2$, we find
\[
\vn{\gamma}_\infty(t) \le C(1 + e^{\frac{t}{2}})\,.
\]
As noted at the start of the proof, this estimate combined with Lemma
\ref{LMksestimate} and Theorem \ref{TMste} gives global existence for (FE).
\end{proof}

Let us now use the interpolation method to establish uniform bounds for all derivatives of
curvature in $L^2$.

\begin{lem}
Let $\gamma:(-1,1)\times[0,\infty)\rightarrow\R^2$ be a solution to (FE)
given by Theorem \ref{TMste}, satisfying \eqref{EQvanillaEhypothesis00}.
Then $T=\infty$ and there exists absolute constants $c_l$ such that
\[
	\vn{k_{s^l}}_2^2 \le c_l
\,.
\]
\label{LMhigher1}
\end{lem}
\begin{proof}
Using Lemma \ref{LMevo2} for the evolution of $k_{s^l}$ and Lemma \ref{LMbdy} to eliminate boundary terms, we find
\begin{align*}
	\frac{d}{dt}\int_\gamma k_{s^l}^2\,ds
	&=
	\frac12\int_\gamma 4k_{s^l}\partial_t k_{s^l} + k_{s^l}^2(2kk_{ss} + k^4)\,ds
	\\&=
	\frac12\int_\gamma k_{s^l}^2(2kk_{ss} + k^4)\,ds
	- 2\int_\gamma k_{s^l}k_{s^{(l+4)}} \,ds
	\\&\qquad
	+ \sum_{q+r+u=l} c_{qru}\int_\gamma k_{s^l}k_{s^{(q+2)}}k_{s^r}k_{s^u} \,ds
	\\&\qquad
 	+ \sum_{q+r+u+v+w=l} c_{qruvw} \int_\gamma k_{s^l}k_{s^q}k_{s^r}k_{s^u}k_{s^v}k_{s^w}\,ds
	\\&=
	- 2\int_\gamma k_{s^{(l+2)}}^2\,ds
	+ \sum_{q+r+u=l} c_{qru}\int_\gamma k_{s^l}k_{s^{(q+2)}}k_{s^r}k_{s^u} \,ds
	\\&\qquad
 	+ \sum_{q+r+u+v+w=l} c_{qruvw} \int_\gamma k_{s^l}k_{s^q}k_{s^r}k_{s^u}k_{s^v}k_{s^w}\,ds
	\,.
\end{align*}
Interpolating again with \cite[Proposition 2.5]{DKS02} (we note again that \cite[Appendix C]{dall2014}
can be consulted for full details), we find
\[
	\frac{d}{dt}\int_\gamma k_{s^l}^2\,ds
	+ \int_\gamma k_{s^{(l+2)}}^2\,ds
	\le C
\]
where $C$ depends on the upper and lower bounds for length, as well as the bound for curvature.
As before, using Lemma 2.7 and Lemma 2.8 implies the differential inequality
\[
	\frac{d}{dt}\int_\gamma k_{s^l}^2\,ds
	+ \hat{c}\int_\gamma k_{s^{(l)}}^2\,ds
	\le \hat{C}
\]
which implies $\vn{k_{s^l}}_2^2$ is uniformly bounded on $[0,\infty)$, as required.
\end{proof}

It remains to show the convergence for (FE).
For this, our strategy is to use the gradient flow structure to show that
eventually the curvature decays exponentially fast. Then, this exponential
decay gives an a-priori uniform bound also for $\vn{\gamma}_\infty$.
Not only this, we will use the exponential decay to obtain precise uniform
control on $\gamma$ as a map in the classical function space
$C^\infty((-1,1),\R^2)$, and in so doing, obtain convergence of $\gamma$ to a
straight line segment parallel to $e$.

The key idea is that the flow must become eventually very close to a straight line.
Once close to a straight line, decay estimates can be proved.
To begin, we show that the curvature must vanish along the flow as $t\rightarrow\infty$.
This uses global existence, the gradient flow structure, and the classification of elastica in
the plane.

\begin{prop}
Let $\gamma:(-1,1)\times[0,T)\rightarrow\R^2$ be a solution to (FE)
given by Theorem 1.1.
Assume that
\begin{equation*}
L(\gamma_0)\int_\gamma k^2\,ds\bigg|_{t=0} \le \pi\,.
\end{equation*}
Then
\begin{equation}
\label{EQdecay1}
\vn{k}_\infty(t)\rightarrow0\qquad \text{as} \qquad t\rightarrow\infty\,.
\end{equation}
\end{prop}
\begin{proof}
Recall $\Ko(t) = L(t)\vn{k}_2^2(t)$.
We begin with the following claim:

\begin{equation}
\label{EQclaim}
\text{{\bf Claim.}
 For each $t\ge0$, the quantity $\Ko(t)$ is strictly decreasing or equal to zero.}
\end{equation}

We already know from Lemma 3.2 that length is weakly monotonically decreasing.
To show the claim, we will prove that for each $t>0$, either $\gamma(\cdot,t)$ is a horizontal line segment and the quantity is equal to zero, or $\frac{d}{dt}\vn{k}_2^2(t) < 0$.
Fix $t\ge0$.
We have
\[
\frac{d}{dt} \vn{k}_2^2(t) = -2\int_\gamma F^2\,ds\,.
\]
Suppose for the purpose of contradiction that $\frac{d}{dt} \vn{k}_2^2(t) = 0$ and $\gamma(\cdot,t)$ is not a horizontal line segment.

If $\frac{d}{dt}\vn{k}_2^2 = 0$ then, by smoothness, we must have $F=0$.
This implies that the curve $\gamma(\cdot,t)$ is a \emph{free elastica}.
Free elastica have been classified, see for example \cite{Linner}.
There are only two possible shapes for free elastica that fit the boundary
condition of meeting two parallel lines perpendicularly with $k_s(0) = k_s(L) =
0$: horizontal line segments and the rectangular elastica.
We have already assumed that $\gamma(\cdot,t)$ is not a horizontal line segment, so it must be a rectangular elastica.

Due to the boundary condition, the rectangular elastica consists of an integer multiple $m\ge1$ of half-periods, with each
half period turning through a total angle of $\pi$.
Then
\[
\int_\gamma |k^{\gamma(\cdot,t)}|\,ds^{\gamma(\cdot,t)} \ge m\pi \ge \pi
\]
so, by H\"older's inequality,
\[
\Ko(t) = L({\gamma(\cdot,t)}) \int_\gamma (k^{\gamma(\cdot,t)})^2\,ds^{\gamma(\cdot,t)} \ge 
\bigg(\int_\gamma |k^{\gamma(\cdot,t)}|\,ds^{\gamma(\cdot,t)}\bigg)^2 \ge \pi^2
\,.
\]
But we already know that $\Ko(t) \le \pi$, so we have a contradiction.
Thus $\gamma(\cdot,t)$ can't be a rectangular elastica either, and we conclude the claim \eqref{EQclaim}.

The claim implies that for each $t_0>0$ there exists a $\delta(t_0) > 0$ such that
\begin{equation}
\label{EQnew2}
\Ko(t) \le \pi-\delta(t_0) \quad\text{for all $t\ge t_0$.}
\end{equation}
Now, the gradient flow structure and global existence implies
\[
	\int_0^\infty \int_\gamma F^2\,ds\,dt \le \frac12\vn{k}_2^2\big|_{t=0}
\,.
\]
Furthermore, the a-priori estimates from Lemma 3.7 imply that $\Big|
\frac{d}{dt} ||F||_2^2(t) \Big| \le C$, and so by Barbalat's lemma we have
$||F||_2^2(t)\rightarrow0$ as $t\rightarrow\infty$.
Therefore
\[
\bigg| \int_\gamma Fk\,ds \bigg| \le \vn{F}_2\vn{k}_2 \rightarrow 0
\]
so integration by parts implies
\[
\int_\gamma Fk\,ds = -\vn{k_s}_2^2 + \frac12\vn{k}_4^4
\]
keeping in mind that $k_s(0) = k_s(L) = 0$.
Therefore we have
\begin{equation}
\Big|-\vn{k_s}_2^2 + \frac12\vn{k}_4^4\Big| \rightarrow 0
\,.
\label{EQnew11}
\end{equation}
Now, \eqref{EQnew2} and Corollary 2.9 implies that
\begin{align}
\frac12\vn{k}_4^4 \le \frac12\vn{k}_\infty^2\vn{k}_2^2
	&\le \frac{\pi-\delta(1)}{2L} \vn{k}_\infty^2\,,\quad\text{ for $t\ge1$}
\notag\\
	&\le \frac{\pi-\delta_1}{2L} \frac{2L}{\pi} \vn{k_s}_2^2
	= (1-\delta_1/\pi) \vn{k_s}_2^2\,.
\label{EQnew3}
\end{align}
Note that here we set $\delta_1 = \delta(1)$ where $\delta$ is the function in \eqref{EQnew2}.
This implies
\[
-\vn{k_s}_2^2 + \frac12\vn{k}_4^4
\le -\vn{k_s}_2^2 + (1-\delta_1/\pi) \vn{k_s}_2^2
= -\frac{\delta_1}{\pi} \vn{k_s}_2^2
\]
and so
\[
\frac{\delta_1}{\pi} \vn{k_s}_2^2
\le \vn{k_s}_2^2 - \frac12\vn{k}_4^4
\,.
\]
Combining this now with \eqref{EQnew3} we have
\[
\lim_{t\rightarrow\infty}\frac{\delta_1}{\pi} \vn{k_s}_2^2(t)
\le \lim_{t\rightarrow\infty} \bigg|\vn{k_s}_2^2(t) - \frac12\vn{k}_4^4(t)\bigg|
= 0\,.
\]
In particular, Corollary 2.9 now implies that $\vn{k}_\infty(t)\rightarrow0$ as $t\rightarrow\infty$, which is the desired result.
\end{proof}

In order to conclude convergence to some horizontal line segment (recall that
horizontal here means parallel to $e$), we need a time-independent bound for the position vector.
Our passage to this estimate will be through exponential decay.
The key decay estimate is below.

\begin{prop}
Let $\gamma:(-1,1)\times[0,T)\rightarrow\R^2$ be a solution to (FE)
given by Theorem 1.1.
Assume that
\begin{equation*}
L(\gamma_0)\int_\gamma k^2\,ds\bigg|_{t=0} \le \pi\,.
\end{equation*}
Then there exist constants $C_0$ and $\delta_0$ such that
\begin{equation}
\label{EQdecay3}
\vn{k}_2^2(t) \le C_0e^{-\delta_0 t}\quad\text{for all $t\ge0$}\,.
\end{equation}
\end{prop}
\begin{proof}
Returning again to the variational structure, we estimate using the H\"older inequality and Lemma 2.8:
\begin{align*}
\frac{d}{dt}\vn{k}_2^2
 &= -2\int_\gamma F^2\,ds = -2\vn{k_{ss}}_2^2 - \frac12\vn{k}_6^6 - 2\int_\gamma k_{ss}k^3\,ds
\\
 &= -2\vn{k_{ss}}_2^2 - \frac12\vn{k}_6^6 + 6\int_\gamma k^2_{s}k^2\,ds
\\
 &\le -2\vn{k_{ss}}_2^2 - \frac12\vn{k}_6^6 + 6\vn{k}_\infty^2\int_\gamma k^2_{s}\,ds
\\
 &\le -2\vn{k_{ss}}_2^2 - \frac12\vn{k}_6^6 + \frac{6L^2}{\pi^2}\vn{k}_\infty^2\int_\gamma k^2_{ss}\,ds
\\
 &\le -(2-6L^2\vn{k}_\infty^2/\pi^2)\vn{k_{ss}}_2^2
\,.
\end{align*}
The convergence in \eqref{EQdecay1} implies that there exists a $t_1\in(0,\infty)$ such that for all $t>t_1$, $\vn{k}_\infty(t) \le \pi/(L_0\sqrt6)$, so that we have (keeping in mind the length bound)
\begin{equation}
\label{EQdecay2}
\frac{d}{dt}\vn{k}_2^2(t)
 \le -\vn{k_{ss}}_2^2(t)\,,\qquad\text{ for $t>t_1$}\,.
\end{equation}
Now we use Lemmata 2.7 and 2.8 to estimate
\[
\frac{L^4}{\pi^4}\vn{k_{ss}}_2^2
\ge
\frac{L^2}{\pi^2}\vn{k_s}_2^2
\ge \vn{k}_2^2\,.
\]
Combining this with \eqref{EQdecay2} and the length estimate we find
\[
\frac{d}{dt}\vn{k}_2^2(t)
\le -\frac{\pi^4}{L^4(t)}\vn{k}_2^2(t)
\le -\frac{\pi^4}{L_0^4}\vn{k}_2^2(t)
\,.
\]
Integration yields the decay estimate
\begin{equation}
\label{EQdecayyy}
\vn{k}_2^2(t) \le \vn{k}_2^2(t_1)e^{-\frac{\pi^4}{L_0^4}(t-t_1)}\qquad\text{ for $t>t_1$}\,.
\end{equation}
Now as $\sup_{t\in(0,t_1)}\vn{k}_2^2(t) \le \vn{k}_2^2(0)$, we may conclude the decay estimate \eqref{EQdecay3} with
\[
C_0 = \vn{k}_2^2(0)e^{\frac{\pi^4}{L_0^4}t_1}\qquad\text{ and }\qquad \delta_0 = \frac{\pi^4}{L_0^4}\,.
\]
\end{proof}

Next, we use the exponential decay from \eqref{EQdecay3} and a-priori estiamtes in Lemma 3.7 to conclude exponential decay of all terms $\vn{k_{s^l}}_2^2(t)$, where $l\ge0$.

\begin{lem}
Let $\gamma:(-1,1)\times[0,T)\rightarrow\R^2$ be a solution to (FE)
given by Theorem 1.1 satisfying \eqref{EQvanillaEhypothesis00}.
Let $l>0$ be in integer.
There exists a $C_l$ such that
\begin{equation}
\label{EQdecay4}
\vn{k_{s^l}}_2^2(t) \le C_le^{-\frac{\delta_0}{2} t}\quad\text{for all $t\ge0$}\,.
\end{equation}
\end{lem}
\begin{proof}
Observe that we may integrate by parts to obtain
\[
\vn{k_{s^l}}_2^2 = [k_{s^l}k_{s^{l-1}}]_{\{0,L\}}-\int_\gamma k_{s^{l-1}}k_{s^{l+1}}\,ds
                 = -\int_\gamma k_{s^{l-1}}k_{s^{l+1}}\,ds
\,,
\]
because no matter what $l>0$ is, at least one of $l$ or $(l-1)$ are odd, and so the boundary terms vanish by Lemma 2.6.
Continuing in this way (the boundary term each time is
$k_{s^{l-k-1}}k_{s^{l+k}}$, and one of $(l-k-1)$ or $(l+k)$ must be odd), we
integrate by parts a total of $2l$ times and then using the H\"older inequality
yields
\[
\vn{k_{s^l}}_2^2(t) \le \vn{k}_2(t) \vn{k_{s^{2l} }}_2(t)
\,.
\]
The decay estimate \eqref{EQdecay4} now follows from \eqref{EQdecay3} and Lemma 3.7.
\end{proof}

A slightly specialised version of the above lemma will be useful for our argument below.

\begin{cor}
Let $\gamma:(-1,1)\times[0,T)\rightarrow\R^2$ be a solution to (FE)
given by Theorem 1.1 satisfying \eqref{EQvanillaEhypothesis00}.
Then, for each $l\in\N$,
\begin{equation}
\label{EQdecay6}
\vn{F_{s^l}}^2_\infty(t) + \vn{k_{s^l}}^2_\infty(t) \le C_le^{-\frac{\delta_0}{2} t}
\end{equation}
\end{cor}
\begin{proof}
In this proof we assume throughout that $t\in(\hat t,\infty)$.
The result for all $t$ will follow by adjusting the constant $C_l$ to account for a compact interval of earlier time (during which the quantity in question is bounded).

First, 
\begin{equation}
\label{EQhelp}
F_{s^l} = k_{s^{l+2}} + \frac12(k^3)_{s^l} 
                = k_{s^{l+2}} + \sum_{q+r+u=l} c_{qru} k_{s^q}k_{s^r}k_{s^u}
\end{equation}
for constants $c_{qru}\in\R$ with $q,r,u\ge0$.
Now, \eqref{EQdecay4} and Corollary 2.9 implies, for any $j\ge0$,
\begin{equation}
\label{EQest1}
	\vn{k_{s^j}}_\infty^2 \le \frac{2L}{\pi}\vn{k_{s^{j+1}}}_2^2
\le C_je^{-\frac{\delta_0}{2} t}
	\,.
\end{equation}
The estimate \eqref{EQest1} applies to each term in \eqref{EQhelp} and this yields the result.
\end{proof}

We are now ready to finish the convergence proof for the flow (FE).
A standard argument (see \cite[Appendix]{ideal}) gives convergence under general conditions.
A point of difference in our argument here is that we do not reparametrise, and prove non-degeneracy of the parametrisation in the limit.
Although this argument is mostly standard and well-known, for the convenience of the reader, we give the details here.

\begin{thm*}[Theorem 1.4]
Suppose $|e|>0$.
Let $\gamma:(-1,1)\times[0,T)\rightarrow\R^2$ be a solution to (FE).
Assume \eqref{EQvanillaEhypothesis1}.
Then the flow exists globally $T=\infty$, and $\gamma(\cdot,t)$ converges
exponentially fast to a straight line segment parallel to $e$ in the $C^\infty$
topology.
\end{thm*}
\begin{proof}
First, note that \eqref{EQevolareaelement} implies
\begin{equation}
\label{EQhelp2}
	(\log |\gamma_x|)_t(x,t) = (kF)(x,t)
\end{equation}
For each fixed $x$ this is an ODE.
Then, \eqref{EQdecay6} implies that $|(\log |\gamma_x|)_t| \le Ce^{-\frac{\delta_0}2t}$.
Set $\hat c_0 = 2C\delta_0^{-1}$.
We find
\begin{equation}
\label{EQhelp25}
|\gamma_x|(x,0)e^{-\hat c_0} \le |\gamma_x|(x,t) \le e^{\hat c_0}|\gamma_x|(x,0)
\,,
\end{equation}
in particular there exists a $c_0$ such that for all $(x,t)$ we have $c_0^{-1} \le |\gamma_x|(x,t) \le c_0$.

This uniform control can be used in an induction argument to convert control on
any number of arc-length derivatives $\partial_s$ to control on parameter
derivatives $\partial_x$.
Indeed, for any map $\phi:(-1,1)\rightarrow\R^r$, $r\in\N$, we have 
\begin{equation}
\label{EQhelp3}
\phi_{x^m}
 =
|\gamma_x|^{m}
\phi_{s^m}
 + 
|\gamma_x|^{m-1} \sum_{p_1+\cdots+p_{m} = m} \phi_{x^{p_1}} (|\gamma_x|^{-1})_{x^{p_2}} \cdots (|\gamma_x|^{-1})_{x^{p_m}}
\end{equation}
where $p_1\in\{1,\ldots,m-1\}$, $p_i\in\{0,\ldots,m-1\}$ for $i\ne1$ and $m\ge1$ is an integer.
We apply this now with $\phi = kF$ to obtain uniform control of $|\gamma_x|_{x^l}$ and $(kF)_{x^l}$ for any $l\ge0$.
Above we show the desired control for $l=0$.
Assuming that there exists $\{c_1, \ldots, c_l\}$ such that 
\begin{equation}
\label{EQhelp26}
(|\gamma_x|^{-1})_{x^l}(x,t) \le c_l\qquad\text{ and }\qquad |(kF)_{x^l}|(x,t) \le c_le^{-\frac{\delta_0}2t}\quad\text{ for all $(x,t)$,}
\end{equation}
we use \eqref{EQevolareaelement} to find
\begin{align*}
	((|\gamma_x|^{-1})_{x^{l+1}})_t&(x,t) 
	= -(|\gamma_x|^{-1}kF)_{x^{l+1}}(x,t)
\\
	&= -(|\gamma_x|^{-1})_{x^{l+1}}(x,t)\,(kF)(x,t)
	  - \sum_{q_1+q_2 = l+1} c_{q_1q_2} (|\gamma_x|^{-1})_{x^{q_1}}(x,t)\, (kF)_{x^{q_2}}(x,t)
\,,
\end{align*}
where $q_1\in\{0,\ldots,l\}$ and $q_2\in\{1,\ldots,l+1\}$.
The inductive hypothesis implies that this equation is an ODE with structure
\[
\Phi'(t) + \Phi(t) \eta_0(t) \le C\sum_{i=1}^{l+1} \eta_i(t)
\]
where $C = C(c_1,\ldots,c_{l})$, $\Phi(t) = (|\gamma_x|^{-1})_{x^{l+1}}(x,t)$, and $\eta_i(t) = |(kF)_{x^i}|(x,t)$.
Therefore
\begin{equation}
\label{EQhelp4}
(e^{-\int_0^t\eta_0(u)\,du}\, \Phi(t))' \le C\sum_{i=1}^{l-1} \eta_i(t)
\,.
\end{equation}
The inductive hypothesis controls $\eta_i(t)$ for $i\in\{0,\ldots,l\}$. For $\eta_{l+1}$, we use \eqref{EQhelp3} to find
\begin{equation*}
\eta_{l+1}
 \le
|\gamma_x|^{l+1}
|(kF)_{s^{l+1}}|
 + 
|\gamma_x|^{l} \sum_{p_1+\cdots+p_{l+1} = l+1} \eta_{p_1} |(|\gamma_x|^{-1})_{x^{p_2}}| \cdots |(|\gamma_x|^{-1})_{x^{p_{l+1}}}|
\end{equation*}
where $p_1\in\{1,\ldots,l\}$, $p_i\in\{0,\ldots,l\}$ for $i\ne1$.
Estimating the right hand side using the inductive hypothesis and \eqref{EQhelp25} reveals that
\[
\eta_{l+1}(t) \le \hat c_{l+1}e^{-\frac{\delta_0}2t}
\]
where $\hat c_{l+1} = \hat c_{l+1}(c_0,\ldots,c_l)$.

Thus for any $t$, $||\eta_i||_{L^1(0,t)} \le 2c_i\delta_0^{-1}$, $i\in\{0,\ldots,l+1\}$.
Using this with \eqref{EQhelp4} we find
\[
\Phi(t) \le \Phi(0) + C
\]
where $C$ now additionally depends on $\delta_0$.
Taking $c_{l+1} = \max\{\hat c_{l+1}, \Phi(0) + C\}$ yields the inductive step, and finishes the proof of the estimates
\eqref{EQhelp26}.

Now, we claim that
\begin{equation}
\label{EQhelp5}
||\gamma_{x^l}||_{L^\infty(-1,1)} \le d_l\,,\text{for $l\in\{0,1\}$ and }
||\gamma_{x^l}||_{L^\infty(-1,1)} \le d_le^{-\frac{\delta_0}{4}t}\,,\text{for $l\ge2$.}
\end{equation}
for constants $d_l$. This follows already for $l=1$ from \eqref{EQhelp25} and
for $l\ge2$ from expanding $\gamma_{s^l} = \kappa_{s^l} = (k\nu)_{s^l}$, using the Frenet-Serret frame equations, and
applying \eqref{EQdecay6}, \eqref{EQhelp3}, \eqref{EQhelp26}.

For $l=0$ we return to the flow equation $\gamma_t(x,t) = -(F\nu)(x,t)$.
Since $|F\nu| \le |F| \le |k_{ss}| + \frac12|k|^3 \le ce^{-\frac{\delta_0}{4}t}$, we may integrate to find
$|\gamma(x,t)| \le |\gamma_0(x)| + 4c\delta_0^{-1} =: d_0$, which settles \eqref{EQhelp5}.

We use now the completeness of $C^\infty((-1,1),\R^2)$.
We claim that for any $\{t_i\}\subset(0,\infty)$, $t_i\rightarrow\infty$, the sequence $\{\gamma(\cdot,t_i)\}$ satisfies
\begin{equation}
\label{EQhelp6}
\lim_{i,j\rightarrow\infty} ||\gamma(\cdot,t_i) - \gamma(\cdot,t_j)||_{C^q(K,\R^2)} = 0 
\end{equation}
for all $q\in\N$ and non-empty $K\subset\subset(-1,1)$. 
If this holds, completeness implies that there exists a unique
$\gamma_\infty\in C^\infty((-1,1),\R^2)$ such that
$\gamma(\cdot,t)\rightarrow\gamma_\infty$ in $C^\infty((-1,1),\R^2)$, which
would finish the proof.

To prove \eqref{EQhelp6}, we use the flow equation and the estimates above.
Indeed, let $K$ be a compact non-empty subset of $(-1,1)$.
For any $q\in\N$, $x\in K$, $t_i,t_j\in(0,\infty)$ with $t_i<t_j$ we have
\begin{align*}
\left|
	\gamma_{x^q}(x,t_j) - \gamma_{x^q}(x,t_i)
\right|
	&= \left|
		\int_{t_i}^{t_j} (F\nu)_{x^q}(x)\,dt
	\right|
\\
	&\le C(q,\delta_0)e^{-\frac{\delta_0}4 t_i}
\end{align*}
where we used \eqref{EQhelp26}.
This estimate implies \eqref{EQhelp6}, and finishes the proof.
\end{proof}

\begin{rmk}[\L{}ojasiewicz-Simon]
Once we have subconvergence, there are by now several ways to upgrade to full
convergence.
The approach of Simon \cite{Simon83}, using in particular the
\L{}ojasiewicz-Simon inequality, can be used to obtain full convergence of the
flow.
Simon's original application was to study the asymptotics of second-order
quasilinear parabolic systems on Riemannian manifolds with application to
minimal submanifold theory. However the technique has incredible generality
(Simon's work is itself an infinite-dimensional generalisation of
\L{}ojasiewicz's work on analytic functions) and is by now used in all areas of
parabolic PDE, including our setting here of fourth-order problems with
boundary.

For the elastic flow, all candidate limits have zero energy. They are not
separated in standard function spaces, being continuous in the $C^0$-norm for
example.
Movement in a direction perpendicular to $e$ corresponds to a degeneracy in the
differential operator.

There are two strategies to prove the \L{}ojasiewicz-Simon inequality that we
wish to describe.
First, as we have subconvergence to a straight line $\gamma_\infty$, we may (as
in the proof given above) use this line to describe the flow at later times as
the graph of a function $u:[-1,1]\rightarrow\R$, that is,
\[
	(x,u(x)) = \gamma(x)
\]
where $\gamma$ may have undergone reparametrisation by a tangential
diffeomorphism.
The flow reduces to a fourth-order quasilinear scalar PDE with Neumann boundary
conditions.
The \L{}ojasiewicz-Simon inequality has been established for similar PDE in
this setting by for example Rybka-Hoffnlann \cite{Piotr99}.
A similar argument works here.

The other strategy which we wish to describe is due to Chill \cite{ChillJFA},
who provides three generic criteria\footnote{Chill's argument applies more
generally; for full details, we invite the interested reader to consult
\cite{ChillJFA}.}, which we translate here as:
\begin{itemize}
	\item That the energy is analytic;
	\item That the gradient of the energy is analytic;
	\item That the derivative of the gradient of the energy evaluated at zero is Fredholm.
\end{itemize}
Of course, each of the above need to be understood in appropriate function spaces.
Recently, Dall'Acqua-Pozzi-Spener \cite{dalllojasiewicz} have, for a
constrained variation of the elastic flow, proved each of the above with
clamped boundary conditions.
In light of Lemma \ref{LMbdy}, their proof applies also here, and so this
approach again yields the \L{}ojasiewicz-Simon inequality.
We did not use the \L{}ojasiewicz-Simon inequality to obtain convergence for
(FE), as the uniform flat parametrisation control used in the proof of Theorem
1.4 works nicely.
\label{RMKlojasiewicz}
\end{rmk}

\begin{rmk}[Linearisation]
Perhaps the most striaghtforward approach (apart from the one we have taken here) is to use linearisation.
Once a candidate for the limit is identified, we can linearise the flow
expressed as the evolution of a graph around that limit, and study the spectrum
of the resultant linear equation.

Specifically, decay in the curvature implies that for large times we may write
the flow as a fourth-order quasilinear scalar problem for a graph function
$u:(-|e|/2,|e|/2)\times[T_0,\infty)\rightarrow\R$ with Neumann boundary
condition.  (Note that here we have used a tangential diffeomorphism to change
the interval $(-1,1)$ to $(-|e|/2,|e|/2)$.)

In this setting we would apply Hale-Raugel \cite{Hale92} (see also Matano
\cite{Matano78convergence} and Zelenyak \cite{68stabilisation}). We note that a
similar application was recently made in \cite{Escudero15} for another
fourth-order parabolic problem.
	
In our case we would need to check that the linearisation around any
equilibrium point has zero as an eigenvalue with multiplicity at most one. The
conclusion from \cite{Hale92} is that the limit of the flow is then unique.
\end{rmk}

\begin{rmk}[Novaga-Okabe]
A new approach to obtaining full convergence from subconvergence is the
innovative technique of Novaga-Okabe \cite{NO}.
They prove that with certain boundary conditions the set of equilibria are
separated in the plane in the sense of the Hausdorff metric, and then show
that in a very general setting this implies full convergence.

For the flows (CD) and (FE) translation acts continuously on equilibria and is
an invariant of the respective energies.
Therefore this technique does not appear to work in these cases.
However, for other kinds of boundary conditions that rule out such continuous
energy-invariant actions, we expect that their technique may prove efficient.
\end{rmk}


Let us now consider the curve diffusion flow (CD) and establish Theorem \ref{TMmainb}.
We will prove that the flow converges to a straight line.
This equilibrium has zero curvature, and so by Lemma \ref{WN}, any solution
that is asymptotic to a straight line must have $\kav(t) = 0$ for all
$t$.
We indirectly assume that $\omega=0$ by taking on an assumption such as \eqref{EQhyp}.
Then $\Ko(t) = L(t)\vn{k}_2^2(t)$, and a-priori estimates become much simpler.
We begin with the estimate below, that preserves $\Ko(0) \le K$
so long as $K\le \frac{2\pi}{3}$.

\begin{lem}
\label{KoESTcaseb}
Let $\gamma:(-1,1)\times[0,T)\rightarrow\R^2$ be a solution to (CD) given by Theorem
\ref{TMste} with $\omega = 0$.
Then, if
\[
L(0)\vn{k}_2^2(0)
 \le K \le \frac{2\pi}3
\]
we have the uniform a-priori estimate
\[
L(t)\vn{k}_2^2(t)
 \leq K
\]
for all $t\in[0,T)$.
\end{lem}
\begin{proof}
Using Lemma \ref{LMbdy} we have
\begin{align*}
\rd{}{t}\Ko &+ \Ko\frac{\vn{k_s}_2^2}{L} + 2L\vn{k_{ss}}_2^2
\\
 &= 3L\int_{\gamma} (k-\kav)^2k^2_s ds
   + 6\kav L \int_{\gamma} (k-\kav)k_s^2 ds
   + 2\kav^2 L \vn{k_s}^2_2.
\end{align*}
Since $\kav = 0$, this simplifies to
\begin{align*}
\rd{}{t}(L\vn{k}_2^2) &+ \vn{k}_2^2\vn{k_s}_2^2 + 2L\vn{k_{ss}}_2^2
 = 3L\int_{\gamma} k^2k^2_s\, ds
\,.
\end{align*}
This term is estimated by
\[
3L\int_{\gamma} k^2k_s^2 ds
 \le \frac{3L}{\pi}(L\vn{k}_2^2)\vn{k_{ss}}_2^2\,.
\]
Therefore
\begin{align*}
\rd{}{t}(L\vn{k}_2^2) &+ \vn{k}_2^2\vn{k_s}_2^2
 + L\bigg(2-\frac3\pi L\vn{k}_2^2\bigg)\vn{k_{ss}}_2^2
 \le 0
\,,
\end{align*}
which upon integration yields
\begin{align*}
L\vn{k}_2^2(t) &+ \int_0^t\vn{k}_2^2\vn{k_s}_2^2\,d\tau
 \leq L\vn{k}_2^2(0)\,,
\end{align*}
as required.
\end{proof}

Theorem \ref{TMste} tells us that in order to conclude global existence, it is
enough to find an a-priori estimate for the position vector in $L^\infty$ and
the first derivative of curvature in $L^2$.
Our strategy for this is to show first an a-priori decay estimate for
$\vn{k_{ss}}_2$, the velocity in $L^2$.

\begin{lem}
\label{LMkss}
Let $\gamma:(-1,1)\times[0,T)\rightarrow\R^2$ be a solution to (CD) given by Theorem \ref{TMste}.
Assume that
\begin{equation}
\label{EQhypot}
L(0)\vn{k}_2^2(0) < \frac{2\pi}{9}
\,.
\end{equation}
Then we have for some uniform constant $\delta_0 > 0$ the exponential decay
\begin{equation}
\label{EQcasebexpodecaykss}
\int_\gamma k_{ss}^2\,ds\bigg|_t
 \le
\int_\gamma k_{ss}^2\,ds\bigg|_{t=0}\, e^{-\delta_0 t}
\,.
\end{equation}
\end{lem}
\begin{proof}
Let us compute the evolution of $\vn{k_{ss}}_2^2$.
By using the commutator $[\partial_s,\partial_t]$ as in Lemma \ref{LMevo}, we find
\begin{align*}
k_{sst} &= k_{sts} - kk_{ss}^2
\\
 &= (- k_{s^5} - k^2k_{s^3} - 3k_{ss}k_sk)_s - kk_{ss}^2
\\
 &= - k_{s^6} - k^2k_{s^4} - 3k_{s^3}k_sk - kk_{ss}^2 - 2kk_sk_{s^3} - 3k_{ss}(k_{ss}k + k_s^2)
\\
 &= - k_{s^6} - k^2k_{s^4} - 5k_{s^3}k_sk - 4kk_{ss}^2 - 3k_{ss}k_s^2
\,.
\end{align*}
Using this and integration by parts we compute the following equality
\begin{align}
\frac{d}{dt}\int_\gamma &k_{ss}^2\,ds
\notag
\\
 &= \int_\gamma \Big[
    - 2k_{s^6}k_{ss} - 2k^2k_{s^4}k_{ss} - 10k_{s^3}k_sk_{ss}k - 7kk_{ss}^3 - 6k_{ss}^2k_s^2
                \Big]\,ds
\notag
\\
 &= -2\vn{k_{s^4}}_2^2 - 6\vn{k_sk_{ss}}_2^2
    + \int_\gamma \Big[
       - 2k_{s^4}k_{ss}k^2 - 10k_{s^3}k_{ss}k_sk - 7k_{ss}^3k
                  \Big]\,ds
\notag
\\
 &= -2\vn{k_{s^4}}_2^2
	+ \vn{k_sk_{ss}}_2^2
    + \int_\gamma \Big[
        - 2k_{s^4}k_{ss}k^2 
	+ 4k_{s^3}k_{ss}k_sk
                  \Big]\,ds
\notag
\\
 &= -2\vn{k_{s^4}}_2^2
	- 2\vn{k_sk_{s^3}}_2^2
	+ \vn{k_sk_{ss}}_2^2
    - 4\int_\gamma
        k_{s^4}k_{ss}k^2 
                  \,ds
\,.
\label{EQevolkss}
\end{align}
In the above we used Lemma \ref{LMbdy} and the no-flux condition to remove the boundary terms.

The last two terms on the right hand side are controlled by the following estimates.
First,
\begin{align*}
    - 4\int_\gamma
        k_{s^4}k_{ss}k^2 
                  \,ds
	&\le 4\vn{k_{s^4}}_2\bigg(\int_\gamma k_{ss}^2k^4\,ds\Big)^\frac12
\\
	&\le 4\vn{k_{s^4}}_2\bigg(\vn{k}_\infty^4\int_\gamma k_{ss}^2\,ds\Big)^\frac12
\\
	&\le 4\vn{k_{s^4}}_2\bigg(\frac{4L^2}{\pi^2} \vn{k_s}_2^4\int_\gamma k_{ss}^2\,ds\bigg)^\frac12
\\
	&\le 4\vn{k_{s^4}}_2\bigg(\frac{4L^2}{\pi^2} \vn{k}_2^2 \vn{k_{ss}}_2^4\bigg)^\frac12
\\
	&\le 4\vn{k_{s^4}}_2\bigg(\frac{4L^2}{\pi^2} \vn{k}_2^4 \vn{k_{s^4}}_2^2\bigg)^\frac12
	 = \frac{8}{\pi} (L\vn{k}_2^2) \vn{k_{s^4}}_2^2
\,.
\end{align*}
Above we used Corollary 2.9, Lemma 2.8, integration by parts (keeping in mind Lemma \ref{LMbdy}) and the H\"older inequality.
We use a similar technique to estimate on the remaining term:
\[
	\vn{k_sk_{ss}}_2^2
	\le \frac{L}{\pi}\vn{k_{ss}}_2^4 \le \frac{L\vn{k}_2^2}{\pi} \vn{k_{s^4}}_2^2
\,.
\]
Combining the above two estimates with the evolution equation \eqref{EQevolkss} we find
\begin{align*}
\frac{d}{dt}\int_\gamma k_{ss}^2\,ds
	+ 2\vn{k_sk_{s^3}}_2^2
 &\le -\Big(2 - \frac9\pi L\vn{k}_2^2\Big)
	\vn{k_{s^4}}_2^2
\\ &\le -\Big(2 - \frac9\pi L(0)\vn{k}_2^2(0)\Big)
	\vn{k_{s^4}}_2^2
 = -\delta \vn{k_{s^4}}_2^2
\,,
\end{align*}
where $\delta = 2 - \frac9\pi L(0)\vn{k}_2^2(0) > 0$ by hypothesis \eqref{EQhypot} and Lemma 3.13.
Therefore
\[
\frac{d}{dt}\vn{k_{ss}}_2^2 \le -\delta \frac{\pi^4}{L^4} \le -\delta \frac{\pi^4}{L_0^4} = -\delta_0\vn{k_{ss}}_2^2
\]
where $\delta_0 = \delta\frac{\pi^4}{L_0^4}$ and the estimate \eqref{EQcasebexpodecaykss} follows.
\end{proof}

Now we give a crucial time-dependent bound for the position vector.

\begin{lem}
\label{LMgammainfty}
Let $\gamma:(-1,1)\times[0,T)\rightarrow\R^2$ be a solution to (CD) given by Theorem \ref{TMste}.
Under the assumptions of Lemma \ref{LMkss}, for a constant $C \in
(0,\infty)$ depending only on $\gamma_0$ we have
\begin{equation}
\vn{\gamma}^2_\infty \le C+Ce^t\,.
\label{EQlinftygamma}
\end{equation}
\end{lem}
\begin{proof}
We calculate
\begin{equation}
\label{EQlminfty1a}
\frac{d}{dt}\int_\gamma |\gamma|^2\,ds
 = -2\int_\gamma \IP{\gamma}{k_{ss}\nu}\,ds
   + \int_\gamma |\gamma|^2kk_{ss}\,ds\,.
\end{equation}
First, Lemma \ref{LMkss} allows us to uniformly estimate
\begin{equation}
\label{EQlminfty2a}
-2\int_\gamma \IP{\gamma}{k_{ss}\nu}\,ds
 \le \int_\gamma k^2_{ss}\,ds + \int_\gamma |\gamma|^2\,ds
 \le C + \int_\gamma |\gamma|^2\,ds\,.
\end{equation}
Now let us deal with the second term.
Integrating by parts and estimating, we find
\begin{align*}
   \int_\gamma |\gamma|^2kk_{ss}\,ds
 &= -\int_\gamma k_s^2|\gamma|^2\,ds - 2\int_\gamma kk_s\IP{\gamma}{\tau}\,ds
\\
 &\le \int_\gamma k^2\,ds = L^{-1}\Ko\,.
\end{align*}
Combining this with \eqref{EQlminfty1a}, \eqref{EQlminfty2a}, we find
\[
\frac{d}{dt}\int_\gamma |\gamma|^2\,ds
 \le C + \int_\gamma |\gamma|^2\,ds
\]
which implies
\[
\bigg( e^{-t}\int_\gamma |\gamma|^2\,ds \bigg)' \le Ce^{-t}\,. 
\]
Integrating from zero to $t$ yields the estimate
\begin{equation}
\label{EQcomb1}
\int_\gamma |\gamma|^2\,ds
 \le (C+\vn{\gamma}_2^2(0))e^t
\,.
\end{equation}
Set $\gamma_i = \IP{\gamma}{e_i}$.
To conclude the $L^\infty$ estimate we use Corollary 2.9 to find
\begin{equation}
\label{EQcomb2}
\vn{\gamma_i - \gammaiav}_\infty^2
 \le \frac{2L}{\pi}\int_\gamma |\IP{\tau}{e_i}|^2ds
 \le \frac{2L^2}{\pi}\,.
\end{equation}
Since $\gamma_i = \gamma_i - \gammaiav + \gammaiav$ we conclude from the above that
\begin{equation}
\label{EQcomb3}
\vn{\gamma_i}_\infty
= \vn{\gamma_i - \gammaiav + \gammaiav}_\infty
\le \frac{L\sqrt2}{\sqrt\pi} + \vn{\gammaiav}_\infty
= \frac{L\sqrt2}{\sqrt\pi} + \gammaiav
\,.
\end{equation}
For the term $\gammaiav$, we use the H\"older inequality, the $L^2$ estimate \eqref{EQcomb1} and the length bound from below to find
\begin{equation}
\label{EQcomb4}
\gammaiav^2 =  \frac1{L^2}\bigg(\int_\gamma |\IP{\gamma}{e_i}|\,ds\bigg)^2
           \le \frac1{L}\int_\gamma |\gamma|^2\,ds
 	   \le \frac1L(C+\vn{\gamma}_2^2(0))e^t
 	   \le \frac1{|e|}(C+\vn{\gamma}_2^2(0))e^t
\,.
\end{equation}
Here we also use the length bound from below.
Since $|\gamma|^2 = \sum_{i=1}^2 \gamma_i^2$, combining \eqref{EQcomb1}--\eqref{EQcomb4} for $i=1$ and $i=2$ yields \eqref{EQlinftygamma}.
\end{proof}

Now we contradict finite maximal time.

\begin{cor}
Let $\gamma:(-1,1)\times[0,T)\rightarrow\R^2$ be a solution to (CD) given by Theorem \ref{TMste}
satisfying the assumptions of Lemma \ref{LMkss}.
Then $T=\infty$.
\label{LTE}
\end{cor}
\begin{proof}
Suppose $T<\infty$. Then by Theorem \ref{TMste} there does not exist a $D\in(0,\infty)$ such that
\begin{equation}
\label{EQlmlte}
\vn{\gamma}_\infty(t) + \vn{k_s}_2(t) \le D\,.
\end{equation}
for all $t\in[0,T)$.
Lemma \ref{LMkss} and Lemma 2.8 implies
\[
\vn{k_s}_2^2 \le \frac{L^2}{\pi^2}\vn{k_{ss}}_2^2 \le C_1\,.
\]
This deals with the second term.
Lemma \ref{LMgammainfty} implies
\[
\vn{\gamma}_\infty < C_2
\]
where $C_2 = C + Ce^T$.
These estimates imply that \eqref{EQlmlte} holds for $D = C_1+C_2$, yielding the desired contradiction.
Therefore $T=\infty$.
\end{proof}

Now that we have global existence, the exponential decay of Lemma
\ref{LMkss} yields pointwise decay of both $\vn{k}_\infty$ and
$\vn{k_s}_\infty$.

\begin{cor}\label{CoC}
Let $\gamma:(-1,1)\times[0,\infty)\rightarrow\R^2$ be a solution to (CD) given by Theorem \ref{TMste}
satisfying the assumptions of Lemma \ref{LMkss}.
Then $\vn{k}_\infty(t)\rightarrow0$ as $t\rightarrow\infty$ with estimate
\[
	\vn{k}_\infty(t) + \vn{k_s}_\infty \le Ce^{-\frac{\delta_0}2 t}
\]
where $C$ is a constant depending only on the initial data $\gamma_0$.
\end{cor}
\begin{proof}
The exponential decay of $||k_{ss}||_2^2(t)$ from Lemma \ref{LMkss} combined with Corollary 2.9 and Lemma 2.8 implies
\[
	\vn{k}_\infty^2(t) \le \frac{2L}{\pi}\vn{k_s}_2^2(t) \le \frac{2L^3}{\pi^3}\vn{k_{ss}}_2^2(t) \le Ce^{-\delta_0t}
\]
and
\[
	\vn{k_s}_\infty^2(t) \le \frac{L}{\pi}\vn{k_{ss}}_2^2(t) \le Ce^{-\delta_0t}
\]
where we also used the length bound in both estimates. This implies the result.
\end{proof}

With this strong decay in hand, we are able to bound the position vector
uniformly in time and conclude a weak convergence result for the flow.
Note that strictly speaking we do not need this result to prove convergence for
(CD), as we can proceed from the above directly to exponential decay for all
derivatives of curvature (see Corollary 3.19).
However we have included the below direct result as it uses a technique that
may be of interest to the reader and may be applicable to other future
problems.

\begin{prop}
\label{PNnew}
Let $\gamma:(-1,1)\times[0,\infty)\rightarrow\R^2$ be a solution to (CD) given by Theorem \ref{TMste}
satisfying the assumptions of Lemma \ref{LMkss}.
Then there exists a straight line segment $l$ with the following property.
The position vector $\gamma$ is eventually graphical over $l$ with graph function $u$ converging to zero along a subsequence in 
$C^2((-|e|/2,|e|/2))$.
\end{prop}
\begin{proof}
We use Arzel\`a-Ascoli to extra a limit, but for this we need to bound the
position vector $|\gamma|$.
Since $k(x,t)\rightarrow0$ as $t\rightarrow\infty$, there exists a $T_*$ such
that $\gamma:(-1,1)\times(T_*,\infty)\rightarrow\R^2$ admits a graph
representation via a function
$u:(-|e|/2,|e|/2)\times(T_*,\infty)\rightarrow\R^2$, so $\gamma(x,t) = (\phi(x,t),
u(\phi(x,t),t))$, where $\phi$ is a diffeomorphism (reparametrisation).
To see this, let $\theta(s,t)$ be the angle that the tangent vector at $(s,t)$ makes with the vector $e$.
Corollary 2.9 gives the estimate
\[
|\theta(s,t)|^2 \le \frac{L}{\pi}\int_\gamma \theta_s^2\,ds \le \frac{L^2}{\pi}\vn{k}_\infty^2(t)
\]
and so the angle that the tangent vector makes with $e$ is uniformly small (recall again, length is bounded).
In particular, we may choose $T_*$ such that for all $t>T_*$ we have $\vn{k}_\infty^2(t) < \frac{\pi^2}{4L_0^2}$.
Then we have for all $t>T_*$ that $|\theta(s,t)| < \frac\pi2$. This means that
the tangent vector can never be perpendicular to $e$, i.e., the flow is
uniformly graphical for $t>T_*$. This argument with the angle also implies that
the spatial gradient of the graph function $u$ is uniformly bounded and
approaching zero.
Furthermore, using $y$ to denote the spatial variable of $u$, curvature is
given by $u_{yy}/\sqrt{1+u_y^2}^3$.
Therefore the second derivative of $u$ must also approach zero.

Our control so far does not give a $t$-independent estimate on
$u$ itself (equivalent to controlling $|\gamma|$).
To fix this, we proceed as follows.
Set $e^\perp$ to be any unit normal vector to $e$ (there are two possibilities).
Calculate (keep in mind Lemma 2.6)
\begin{align}
\frac{d}{dt}\int_\gamma \IP{\gamma}{e^\perp}^2\,ds
 &= -2\int_\gamma k_{ss}\IP{\nu}{e^\perp}\IP{\gamma}{e^\perp}\,ds
	+ \int_\gamma kk_{ss}\IP{\gamma}{e^\perp}^2\,ds
\notag\\
 &=     - 4\int_\gamma kk_{s}\IP{\tau}{e^\perp}\IP{\gamma}{e^\perp}\,ds
        + 2\int_\gamma  k_{s}\IP{\nu}{e^\perp}\IP{\tau}{e^\perp}\,ds
\notag\\
	&\qquad - \int_\gamma k_{s}^2\IP{\gamma}{e^\perp}^2\,ds
\notag\\
 &=     - 4\int_\gamma kk_{s}\IP{\tau}{e^\perp}\IP{\gamma}{e^\perp}\,ds
        + 2\int_\gamma  k^2 (\IP{\tau}{e^\perp}^2 - \IP{\nu}{e^\perp}^2)\,ds
\notag\\
	&\qquad - \int_\gamma k_{s}^2\IP{\gamma}{e^\perp}^2\,ds
\notag
\,.
\end{align}
Now we use the Cauchy inequality on the first term to estimate
\[
\frac{d}{dt}\int_\gamma \IP{\gamma}{e^\perp}^2\,ds
\le  4\int_\gamma k^2\IP{\tau}{e^\perp}^2\,ds
        + 2\int_\gamma  k^2 (\IP{\tau}{e^\perp}^2 - \IP{\nu}{e^\perp}^2)\,ds
\,.
\]
Then we use the obvious estimate on the unit vectors to see that
\[
\frac{d}{dt}\int_\gamma \IP{\gamma}{e^\perp}^2\,ds
\le  6\int_\gamma k^2\,ds
\,.
\]
Now the exponential decay estimate can be integrated (Corollary 3.16) to find
\[
\int_\gamma \IP{\gamma}{e^\perp}^2\,ds\bigg|_t
	\le \int_\gamma \IP{\gamma}{e^\perp}^2\,ds\bigg|_{t=0} + C\,.
\]
From this estimate it will follow that the position vector is bounded independent of $t$.
First, the eventual uniform graphicality implies that we only need to bound the height $|\IP{\gamma}{e^\perp}|$.
Second, we observe that (Corollary 2.9)
\begin{align*}
\vn{\IP{\gamma}{e^\perp}}_\infty &= 
\vn{\IP{\gamma}{e^\perp} - \overline{\IP{\gamma}{e^\perp}} + \overline{\IP{\gamma}{e^\perp} }}_\infty
\le
\vn{\IP{\gamma}{e^\perp} - \overline{\IP{\gamma}{e^\perp} }}_\infty
 + \vn{\overline{\IP{\gamma}{e^\perp} }}_\infty
\\
&\le \sqrt{\frac{2L}{\pi}} \bigg(\int_\gamma \IP{\tau}{e^\perp}\,ds\bigg)^{\frac12}
	+ \frac1L\int_\gamma \IP{\gamma}{e^\perp}\,ds
\\
&\le C + C\bigg( \int_\gamma \IP{\gamma}{e^\perp}^2\,ds\bigg|_{t=0} + C\bigg)^{\frac12}
\le C
\end{align*}
where here $C$ varies from line to line but is always a universal constant, and
again we note that length is uniformly bounded. We also used the H\"older
inequality, and the triangle inequality.

Thus we find a uniform a-priori estimate on the magnitude of the graph function
$|u(y,t)|$ and the position vector $|\gamma|$ (by projection onto
$\{e,e^\perp\}$ and using eventual graphicality).
Therefore by Arzel\'a-Ascoli a straight line segment $l$ exists to which
$\gamma(\cdot,t)$ is subconvergent, and in particular the graph function $u$
subconverges to a constant $c$ in $C^2((-|e|/2,|e|/2))$.
By replacing $u$ with $u(y,t) - c$, we have that $l$ is the zero level set of
$u$, and $u$ subconverges to zero.  This finishes the proof.
\end{proof}

Global existence and the exponential decay from Corollary \ref{CoC} (actually we only need the boundedness of $\vn{k}_2^2$ that follows already from Lemma \ref{KoESTcaseb})
imply the following uniform estimates for all derivatives of curvature, a (CD) analogue of Lemma \ref{LMhigher1}.

\begin{lem}
Let $\gamma:(-1,1)\times[0,\infty)\rightarrow\R^2$ be a solution to (CD) given by Theorem \ref{TMste}
satisfying the assumptions of Lemma \ref{LMkss}.
For each non-negative integer $l$, there exists absolute constants $d_l$ such that
\[
	\vn{k_{s^l}}_2^2 \le d_l
\,.
\]
\label{LMhighercdf}
\end{lem}
\begin{proof}
This is the (CD) flow analogue to Lemma \ref{LMhigher1} earlier, and so we will be brief.
Lemma \ref{LMevo2} gives the evolution of $k_{s^l}$, and then integration by
parts (with Lemma \ref{LMbdy} to eliminate boundary terms) gives
\begin{align*}
	\frac{d}{dt}\int_\gamma k_{s^l}^2\,ds
	&=
	- 2\int_\gamma k_{s^{(l+2)}}^2\,ds
	+ \sum_{q+r+u=l} c_{qru}\int_\gamma k_{s^l}k_{s^{(q+2)}}k_{s^r}k_{s^u} \,ds
	\\&\qquad
 	+ \sum_{q+r+u+v+w=l} c_{qruvw} \int_\gamma k_{s^l}k_{s^q}k_{s^r}k_{s^u}k_{s^v}k_{s^w}\,ds
	\,.
\end{align*}
Using now \cite[Proposition 2.5]{DKS02} (c.f. \cite[Appendix C]{dall2014}) to interpolate, we find
\[
	\frac{d}{dt}\int_\gamma k_{s^l}^2\,ds
	+ \int_\gamma k_{s^{(l+2)}}^2\,ds
	\le Ce^{-\delta_0t} \le C
\]
where $C$ depends on the upper and lower bounds for length, as well as the bound for $\vn{k}_2^2$.
Using Lemma 2.7 and Lemma 2.8, as well as the length bound again, this implies the differential inequality
\[
	\frac{d}{dt}\int_\gamma k_{s^l}^2\,ds
	+ \hat{c}\int_\gamma k_{s^{(l)}}^2\,ds
	\le \hat{C}
\,.
\]
This gives that $\int_\gamma k_{s^l}^2\,ds$ is uniformly bounded, as required.
\end{proof}

Uniform boundedness, exponential decay of $\vn{k}_\infty$, and global existence now gives exponential decay of $\vn{k_{s^m}}_\infty$ for all $m\ge0$, 
This kind of argument is similar to \cite[Lemma 2.8]{GWconvcdf}.
As this is very much the same as our earlier result Lemma 3.10 for (FE), we omit the proof.

\begin{cor}
Let $\gamma:(-1,1)\times[0,\infty)\rightarrow\R^2$ be a solution to (CD) given by Theorem \ref{TMste}
satisfying the assumptions of Lemma \ref{LMkss}.
For each non-negative integer $l$, there exists an absolute constant $d_l$ such that
\[
	\vn{k_{s^l}}_\infty(t) \le d_le^{-\frac{\delta_0}{4}t}
\,.
\]
\label{CYalldecayexpcdf}
\end{cor}

Now that we have exponential decay of all norms of curvature, we conclude exponential convergence of the flow in $C^\infty$.
As this is very similar to the case of the flow (FE) earlier, we again omit the proof.

\begin{thm}[Theorem \ref{TMmainb}]
Suppose $|e|>0$.
Let $\gamma:(-1,1)\times[0,T)\rightarrow\R^2$ be a solution to (CD).
Suppose $\gamma_0$ satisfies \eqref{EQhyp}.
Then $\omega = 0$, the flow exists globally $T=\infty$, and $\gamma(\cdot,t)$
converges exponentially fast to a straight line segment parallel to $e$ in the
$C^\infty$ topology.
\end{thm}

\bibliographystyle{plain}
\bibliography{ParallelPlanar}

\end{document}